\documentclass{ws-m3as}
\usepackage{amsmath,amssymb}
\usepackage[dvips]{color}
\usepackage{moreverb}
\usepackage{overpic} 
\usepackage{graphicx,psfrag,subfigure}
\usepackage[dvips]{color}
\usepackage{lmodern,pst-node}
\SpecialCoor

\def\H{\boldsymbol{H}}
\def\L{\boldsymbol{L}}
\def\B{\mathbb B}

\def\e{\boldsymbol{e}}
\newcounter{const}
\def\setc#1{\refstepcounter{const}\label{const:#1}C_{\theconst}}
\def\c#1{C_{\ref{const:#1}}}

\def\weakto{\rightharpoonup}
\newcommand{\dual}[2]{\left\langle #1,#2 \right\rangle}
\def\R{\mathbb{R}}

\def\EE{\mathcal{E}}
\def\II{\mathcal{I}}
\def\KK{\boldsymbol{\mathcal{K}}}
\def\MM{\boldsymbol{\mathcal{M}}}
\def\NN{\mathcal{N}}
\def\PP{\mathcal{P}}
\def\SS{\mathcal{S}}
\def\TT{\mathcal{T}}
\def\VV{\boldsymbol{\mathcal{V}}}
\def\v{\boldsymbol{v}}
\def\x{\boldsymbol{x}}
\def\m{\boldsymbol{m}}
\def\n{\boldsymbol{n}}

\def\f{\boldsymbol{f}}
\def\heff{\boldsymbol{h}_{\rm eff}}
\def\ppsi{\boldsymbol\psi}
\def\pphi{\boldsymbol\phi}
\def\Cexchange{C_{\rm exch}}

\def\Cstray field{C_{\rm stray}}
\def\Canisotropy{C_{\rm ani}}

\def\set#1#2{\big\{#1\,:\,#2\big\}}
\def\operator{\boldsymbol{\pi}}
\newcommand{\norm}[3][]{#1\|#2#1\|_{#3}}
\def\revision#1{{\color{black}#1}}

\def\check#1{{\color{black}#1}}

\def\dist{{\rm dist}}
\def\normal{{\boldsymbol{\nu}}}

\def\OO{\mathcal O}
\def\eps{\varepsilon}
\def\A{\mathcal A}
\def\uu{\mathbf u}
\def\vv{\mathbf v}
\def\ww{\mathbf w}

\def\M{\boldsymbol{M}}
\def\BB{\boldsymbol{B}}
\def\j{\boldsymbol{j}}
\def\Hext{\boldsymbol{F}}
\def\Heff{\boldsymbol{H}_{\rm eff}}
\def\uext{u_{\rm app}}

\def\hext{\boldsymbol{f}}

\def\w{\mathbf{w}}
\newtheorem{propappendix}{Proposition}[appendixsection]
\newtheorem{remarkappendix}{Remark}[appendixsection]

\newtheorem{algorithm}{Algorithm}[section]
\begin{document}
\markboth{F.~Bruckner, M.~Feischl, T.~F\"uhrer, P.~Goldenits, M.~Page, D.~Praetorius, \revision{M.~Ruggeri,} and D.~Suess}
{Multiscale Modelling in Micromagnetics: 
\revision{Existence of Solutions} and Numerical Integration}
%
\catchline{}{}{}{}{}
%
\title{Multiscale Modelling in Micromagnetics:\\\revision{Existence of Solutions} and Numerical Integration}
\author{F.~Bruckner, D.~Suess}
\address{Institute of Solid State Physics, Vienna University of Technology, Wiedner Hauptstra\ss{}e 8--10\\
Vienna, 1040, Austria\\
\{Florian.Bruckner,Dieter.Suess\}@tuwien.ac.at}
\author{M.~Feischl, T.~F\"uhrer\footnote{corresponding author}, P.~Goldenits, M.~Page, D.~Praetorius, \revision{M.~Ruggeri}}
\address{Institute for Analysis and Scientific Computing, Vienna University of Technology, Wiedner Hauptstra\ss{}e 8--10\\
Vienna, 1040, Austria\\
\{Michael.Feischl,Thomas.Fuehrer,Dirk.Praetorius,\revision{Michele.Ruggeri}\}@tuwien.ac.at}
\maketitle
\begin{history}
\received{(Day Month Year)}
\revised{(Day Month Year)}
\comby{(xxxxxxxxxx)}
\end{history}
\begin{abstract}
Various applications ranging from spintronic devices, giant magnetoresistance sensors, and magnetic storage devices, include magnetic parts on very different length scales.
Since the consideration of the Landau-Lifshitz-Gilbert equation (LLG) constrains the maximum element size to the exchange length within the media, it is numerically not attractive to simulate macroscopic parts with this approach.
On the other hand, the magnetostatic Maxwell equations do not constrain the element size, but cannot describe the short-range exchange interaction accurately.
A combination of both methods allows to describe magnetic domains within the micromagnetic regime by use of LLG and also considers the macroscopic parts by a non-linear material law using the Maxwell equations.
In our work, we prove that under certain assumptions on the non-linear material law, this multiscale version of LLG admits weak solutions.
Our proof is constructive in the sense that we provide a linear-implicit numerical integrator for the multiscale model such that the numerically computable finite element solutions admit weak $H^1$-convergence (at least for a subsequence) towards a weak solution.
\end{abstract}
\ccode{\revision{AMS Subject Classification: 35K22, 65M60, 65N30}}
\section{Introduction}
\label{sec:intro}
\noindent
The understanding of magnetization dynamics, especially on a microscale, is of utter relevance, for example in the development of magnetic sensors, recording heads, and magnetoresistive storage devices.
In the literature, a well accepted model for micromagnetic phenomena is the Landau-Lifshitz-Gilbert equation (LLG), see~\eqref{eq:llg:physics}.
This non-linear partial differential equation describes the behaviour of the magnetization of some ferromagnetic body under the influence of a so-called effective field.
Existence (and non-uniqueness) of weak solutions of LLG goes back to Ref.~\refcite{as}.
As far as numerical simulation is concerned, convergent integrators can be found, e.g., in the works of Refs.~\refcite{bp}, \refcite{bjp} or~\refcite{bbp}, where even coupling to Maxwell's equations is considered.
For a complete review, we refer to Refs.~\refcite{cimrak}, \refcite{gc}, \refcite{mp06} or the monographs~\refcite{hubertschaefer}, \refcite{prohl} and the references therein.
Recently, there has been a major breakthrough in the development of effective and mathematically convergent algorithms for the numerical integration of LLG.
In Ref.~\refcite{alouges08}, an integrator is proposed which is unconditionally convergent and only needs the solution of one linear system per time step.
The effective field in this work, however, only covers microcrystalline exchange effects and is thus quite restricted.
In the subsequent works of Refs.~\refcite{alouges11}, \refcite{goldenits}, \refcite{mathmod2012}, \refcite{gamm2011} the analysis for this integrator was widened to cover more general (linear) field contributions while still \revision{conserving} unconditional convergence. 
\par In our work, we generalize the integrator from Ref.~\refcite{alouges08} even more and basically allow arbitrary field contributions (Section~\ref{section:general}).
Under some assumptions on those contributions, namely boundedness and some weak convergence property,
see~\eqref{assumption:chi:bounded}--\eqref{assumption:chi:convergence}, our main theorem still proves unconditional convergence towards some weak solution of LLG (Theorem~\ref{theorem}).
In particular, our analysis allows to incorporate the approximate \revision{computation} of effective field contributions like, e.g., the stray field which cannot be computed analytically in practice, but requires certain FEM-BEM coupling methods (Section~\ref{section:fredkinkoehler}).
Such additional approximation errors have so far been neglected in the previous works.
To illustrate this, we show that the hybrid \revision{FEM-BEM approaches from Refs.~\refcite{fredkinkoehler,gcr}} for stray field computations does not affect the unconditional convergence of the proposed integrator (Proposition~\ref{prop:stray field}, Proposition~\ref{prop:gcr}).
\par From the point of applications, the numerical integration of LLG restricts the maximum element size for the underlying mesh to the (material dependent) exchange length in order to numerically resolve domain wall patterns.
Otherwise, the numerical simulation is not able to capture the effects stemming from the exchange term and would lead to qualitatively wrong and even unphysical results.
However, due to limited memory, this constraint on the mesh-size practically also imposes a restriction on the actual size of the contemplated ferromagnetic sample.
Considering the magnetostatic Maxwell equations combined with a (non-linear) material law instead, one does not face such a restriction on the mesh-size (and thus on the computational domain). On the one hand, this implies that such a rough model cannot be used to describe short-range interactions like those driving LLG.
On the other hand, this gives us the opportunity to cover larger domains and still maintain a manageable problem size.
\par In our work, we show how to combine microscopic and macroscopic domains to simulate a multiscale problem (Section~\ref{sec:multimodel}):
On the microscopic part, where we aim to simulate the configuration of the magnetization, we solve LLG.
The influence of a possible macroscopic part, where the magnetization is not the goal of the computation, is described by means of the magnetostatic Maxwell equations in combination with some (non-linear) material law.
This macroscopic part then gives rise to an additional non-linear and nonlocal field contribution (Section~\ref{sec:multi}) such that unconditional convergence of the numerical integrator or even mere existence of weak solutions in this case is not obvious.
For certain practically relevant material laws, we analyze a discretization of the multiscale contribution by means of the Johnson-N\'ed\'elec coupling and prove that the proposed numerical integrator still preserves unconditional convergence.
Striking numerical experiments for our approach are given and discussed in Ref.~\refcite{bruckner}.
\subsection*{Outline}
\noindent The remainder of this paper is organized as follows:
In Section~\ref{sec:multimodel}, we give a motivation and the mathematical modelling for our multiscale model.
While Section~\ref{sec:maxwell} focuses on the new contribution to the effective field, Section~\ref{subsection:llg} recalls the LLG
equation used for the microscopic part.
In Section~\ref{section:general}, we introduce our numerical integrator in a quite general framework and formulate the main result (Theorem~\ref{theorem}) which states unconditional convergence under certain assumptions on the (discretized) effective field contributions.
The remainder of this section is then dedicated to the proof of Theorem~\ref{theorem}.
In Section~\ref{section:heff}, we consider different effective field contributions as well as possible discretizations and show that the assumptions of Theorem~\ref{theorem} are satisfied.
Our analysis includes general anisotropy densities (Section~\ref{sec:anisotropy}) as well as contributions which stem from the solution of operator equations with \revision{strongly} monotone operators (Section~\ref{sec:monotone}).
This abstract framework then covers, in particular, the hybrid FEM-BEM \revision{discretizations from Refs.~\refcite{fredkinkoehler,gcr}} for the stray field (Section~\ref{section:fredkinkoehler}) as well as the proposed multiscale contribution to the effective field (Section~\ref{sec:multi}). \revision{A short appendix comments on some physical energy dissipation.}

\section{Multiscale model}\label{sec:multimodel}
In our model, we consider two separated ferromagnetic bodies $\Omega_1$ and $\Omega_2$ as schematized in Figure~\ref{fig:regions}.
Let $\Omega_1,\Omega_2\subset\R^3$ be bounded Lipschitz domains with Euclidean distance $\dist(\Omega_1,\Omega_2)>0$ and boundaries $\Gamma_1 = \partial\Omega_1$ resp.\ $\Gamma_2=\partial\Omega_2$.
On the microscopic part $\Omega_1$, we are interested in the domain configuration and thus solve \revision{LLG}.
On $\Omega_2$, we will use the macroscopic Maxwell equations with a (possibly non-linear) material law instead.
\par To motivate this setting, we consider a magnetic recording head (see Figures~\ref{fig:regions} and \ref{fig:readhead_overview}).
The microscopic sensor element is based on the  giant magnetoresistance effect \revision{(GMR)}, and it requires the use of LLG in order to describe the short range interactions between the individual layers of the sensor accurately.
On the other hand, the smaller these sensor elements are, the more important becomes the shielding of the stray field of neighbouring data bits.
In practice, this is achieved by means of some macroscopic softmagnetic shields located directly besides the GMR sensor.
Describing these large components by use of LLG would lead to very large problem sizes, because the detailed domain structure within the magnetic shields would be calculated.
As proposed in this paper, macroscopic Maxwell equations allow to overcome this limitation and thus provide a profound method to describe the influence of the shields in an averaged sense.
While this work focuses on the mathematical model and a possible discretization, we refer to Ref.~\refcite{bruckner} for numerical simulations and the experimental validation of the proposed model.
\begin{figure}[h!]
\centering
\begin{overpic}[width=0.5\columnwidth]{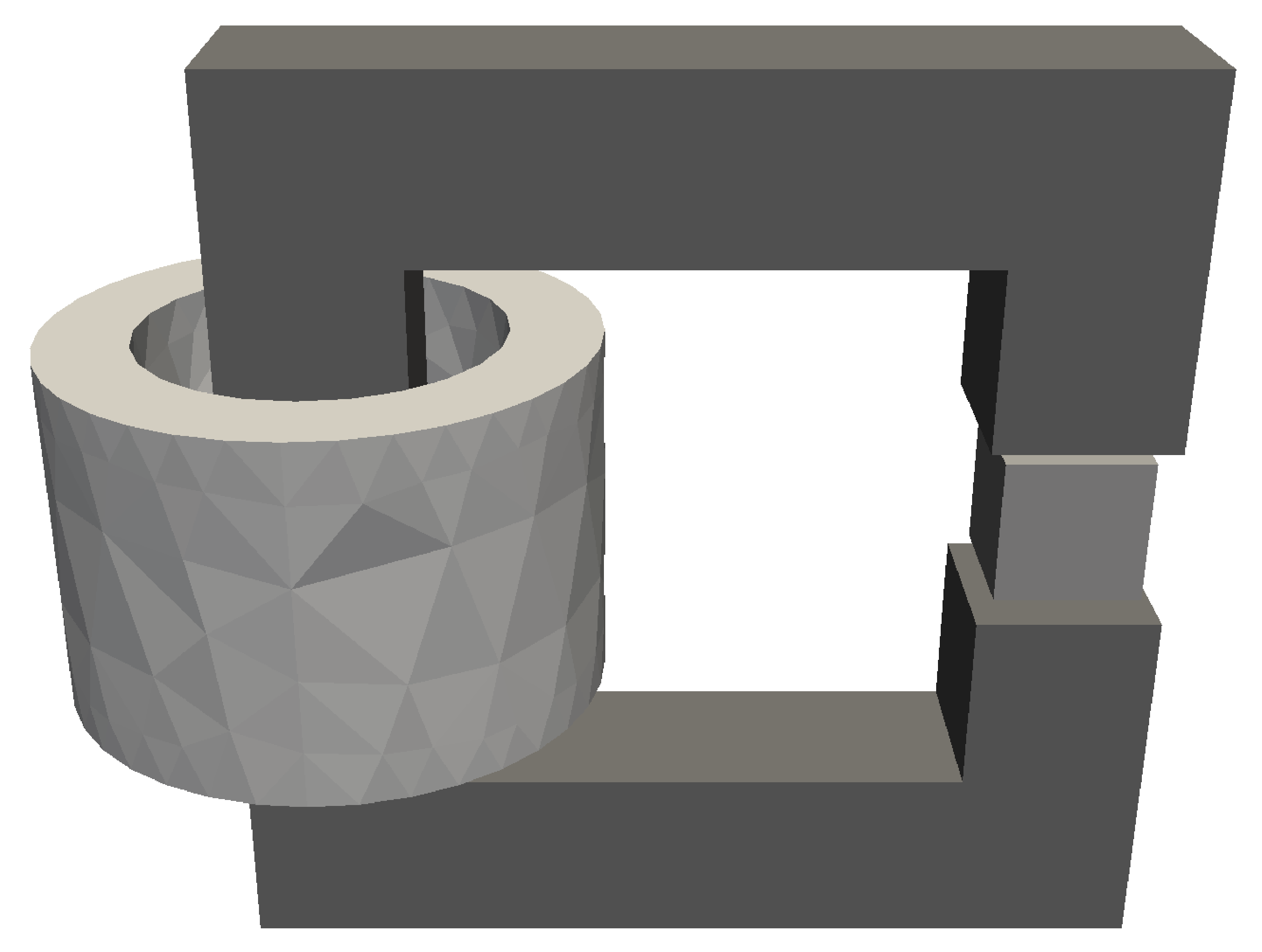}
	\put(20,26){\color{white}$\Omega_{coil}$}
	\put(83,32){\color{white}$\Omega_{1}$}
	\put(50,60){\color{white}$\Omega_{2}$}
\end{overpic}
\caption{\small Example geometry which demonstrates model separation into LLG region $\Omega_1$ and Maxwell region $\Omega_2$ (and in this case in an electric coil region $\Omega_{coil}$). Here, $\Omega_1$ represents one grain of a recording media and $\Omega_2$ shows a simple model of a recording write head.}
\label{fig:regions}
\end{figure}
\begin{figure}[h!]
\centering
\begin{overpic}[width=0.5\columnwidth]{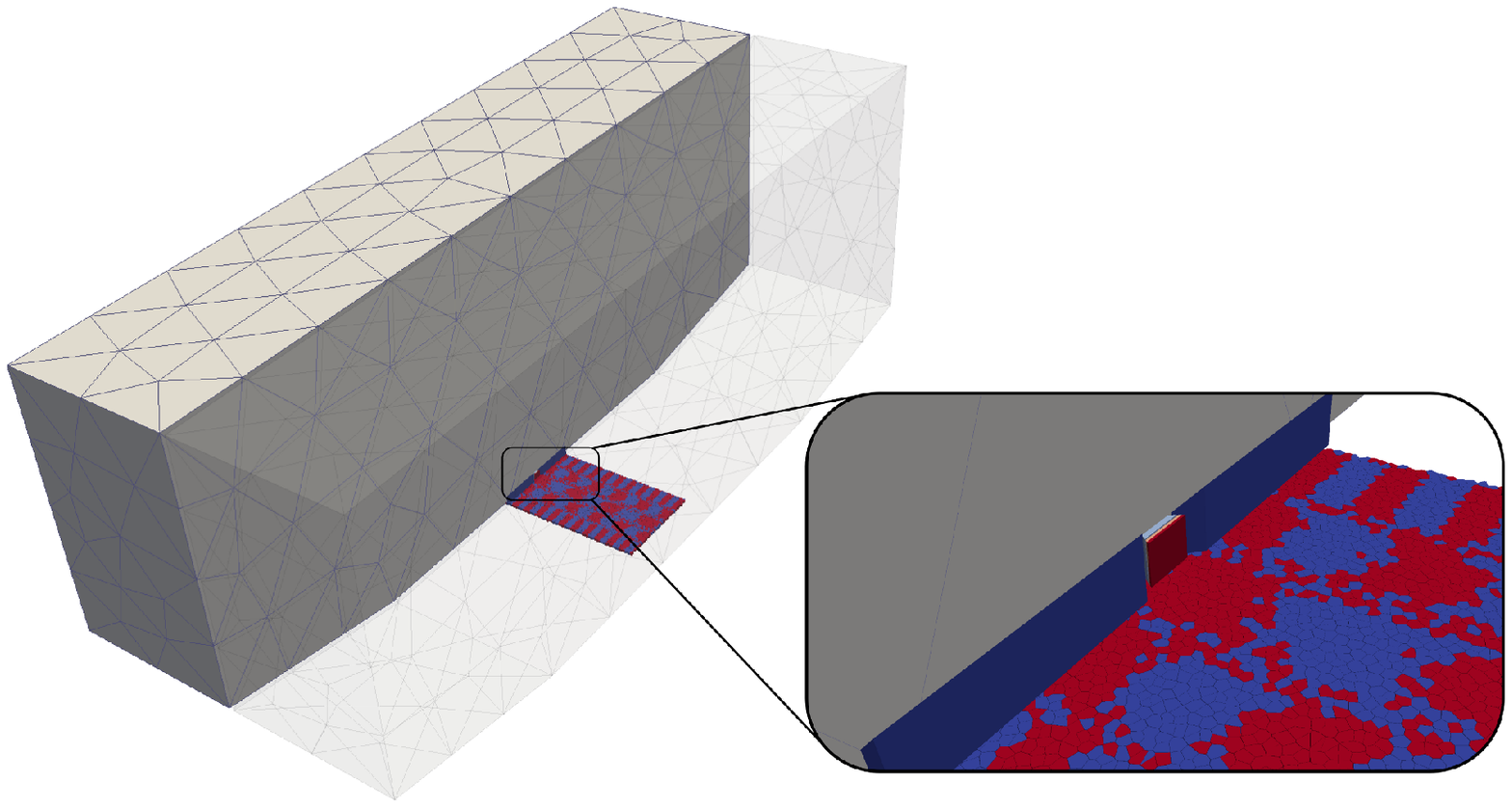}
\end{overpic}
\caption{\small The example setup consists of a microscopic GMR sensor element in between two macroscopic shields. Beyond the GMR sensor a magnetic storage media is indicated. The multiscale algorithm is used to calculate the stationary state of the GMR sensor for various applied external fields.} 
\label{fig:readhead_overview}
\end{figure}
\subsection{Magnetostatic Maxwell equations}\label{sec:maxwell}
The magnetostatic Maxwell equations read
\begin{align}\label{eq:maxwell}
 \nabla\times\H = \j
 \quad\text{and}\quad
 \nabla\cdot\BB = 0
 \quad\text{in }\R^3,
\end{align}
where $\H:\R^3\to\R^3$ is the magnetic field strength $[A/m]$ and $\BB:\R^3\to\R^3$ is the magnetic flux density $[T]$ which are related by
\begin{equation}\label{eq:flux}
\BB = \mu_0(\H + \M) \quad \text{in }\R^3
\end{equation}
with $\mu_0=4\pi\cdot10^{-7}$ $Tm/A$ the permeability of vacuum.
The current density $\j$ $[A/m^2]$ is the source of the magnetic field strength $\H$.
The magnetization field $\M$ $[A/m]$ is non-trivial on the magnetic bodies $\Omega_1\cup\Omega_2$, but vanishes in $\R^3\backslash\overline{(\Omega_1\cup\Omega_2)}$.
The total magnetic field is split into
\begin{align}\label{eq:splitTotalMagField}
\H = \H_1 + \H_2 + \Hext,
\end{align}
where $\H_j:\R^3\to\R^3$ is the magnetic field induced by the magnetization $\M_j=\M|_{\Omega_j}$ on $\Omega_j$ and $\Hext$ is the field generated by the current density $\j$ in $\R^3\backslash\overline{\Omega_1\cup\Omega_2}$.
This implies
\begin{align}
\nabla\times\Hext = \j
\quad\text{and therefore}\quad
\nabla\times\H_j = 0
\quad\text{in }\R^3.
\end{align}
In particular, the induced fields are gradient fields $\H_j = -\nabla U_j$ with certain scalar potentials $U_j:\R^3\to\R$.
We assume that $\Hext$ is induced by currents only, but not by magnetic monopoles. Therefore,
\begin{align}\label{eq:Hext:div}
\nabla\cdot\Hext = 0\quad\text{in }\R^3.
\end{align}
Moreover, the sources of $\H_j$ lie inside $\Omega_j$ only and hence
\begin{align}\label{eq:divHj}
\nabla\cdot\H_j = 0 \quad\text{in }\R^3 \backslash \overline\Omega_j.
\end{align}
From the magnetic flux $\BB$, we obtain
\begin{equation*}
0
= \nabla\cdot\BB = \mu_0(\nabla\cdot\H + \nabla\cdot\M)
= \mu_0(\nabla\cdot\H_j + \nabla\cdot\revision{\M_j})
\quad \text{in } \Omega_j.
\end{equation*}
Together with $\H_j = -\nabla U_j$ and~\eqref{eq:divHj}, this reveals
\begin{subequations}\label{eq:uj}
\begin{align}
\label{eq:uj:interior}
 \Delta U_j &= \nabla\cdot\revision{\M_j} \quad\text{in }\Omega_j,\\
\label{eq:uj:exterior}
 \Delta U_j &= 0 \hspace*{13.0mm}\text{in }\R^3 \backslash\overline\Omega_j.
\end{align}
\end{subequations}
For the micromagnetic body $\Omega_1$, the respective magnetization $\M_1$ is computed by LLG, see Section~\ref{subsection:llg} below.
The overall transmission problem~\eqref{eq:uj} for $\Omega_1$, supplemented by transmission conditions as well as a radiation condition, reads
\begin{subequations}\label{eq:u1}
\begin{align}
 \Delta U_1 &= \nabla\cdot\M_1
 \hspace*{6.65mm}\text{in }\Omega_1,\\
 \Delta U_1 &= 0
 \hspace*{16.5mm}\text{in }\R^3 \backslash\overline\Omega_1,\\
 U_1^{\rm ext}-U_1^{\rm int} &= 0
 \hspace*{16.5mm}\text{on }\Gamma_1,\\
 \revision{\nabla(U_1^{\rm ext}- U_1^{\rm int})\cdot\normal_1} &= -\M_1\cdot
 \revision{\normal_1}
 \hspace*{3.0mm}\text{on }\Gamma_1,\\
 U_1(x) &= \OO(1/|x|)
 \hspace*{5.3mm}\text{as }|x|\to\infty.
\end{align}
\end{subequations}
Here, the superscripts \emph{int} and \emph{ext} indicate whether the trace is considered from inside $\Omega_1$ (resp. $\Omega_2$ in~\eqref{eq:u:omega2} below) or the exterior domain $\R^3\backslash\overline\Omega_1$ (resp. $\R^3\backslash\overline\Omega_2$ in~\eqref{eq:u:omega2} below).
Moreover, \revision{$\normal_j$} denotes the outer unit normal vector on $\Gamma_j$, which points from \revision{$\Omega_j$} to the exterior domain \revision{$\R^3\backslash\overline\Omega_j$}.
For the macroscopic body $\Omega_2$, we assume a non-linear material law
\begin{align}\label{eq:material}
 \revision{\M_2} = \chi(|\H|)\H
 \quad\text{on }\Omega_2
\end{align}
with a scalar function $\chi:\R_{\ge0}\to\R$ and $|\cdot|$ the modulus. Some examples for suitable $\chi$ are listed below (see Remark~\ref{rem:materiallaw}).
\par For the computation of the potential $U_2$, we introduce an auxiliary potential $U_{\text{app}}$. 
\revision{Since} $\nabla \times \Hext = 0$ in \revision{the simply connected
domain $\Omega_2$}, we infer $\Hext = -\nabla U_{\text{app}}$ on $\Omega_2$ with some
potential $U_{\text{app}} : \Omega_2\rightarrow \R$. According to~\eqref{eq:Hext:div} and up to an additive constant,
\revision{$U_{\text{app}}$} can be obtained as the unique solution of the Neumann problem
\begin{subequations}\label{eq:uext}
\begin{align}\label{eq:uext:interior}
  \Delta U_{\text{app}} &= 0 \hspace*{19.3mm}\text{in }\Omega_2, \\
  \label{eq:uext:boundary}
  \revision{\nabla U_{\text{app}}^{\rm int}\cdot\normal_2} &= - \Hext^{\rm int}\cdot\revision{\normal_2} 
  \hspace*{5mm}\text{on }\Gamma_2,
\end{align}
\end{subequations}
with $\int_{\Omega_2} U_{\text{app}} = 0$. The transmission problem for the total
potential $U = U_1 + U_2 + U_{\text{app}}$ of the total magnetic field $\H = -\nabla U$ 
in $\Omega_2$ and for the potential $U_2$ in
$\R^3\backslash\overline\Omega_2$, supplemented by a radiation condition,
reads
\begin{subequations}\label{eq:u:omega2}
\begin{align}
  \label{eq:u:omega2:interior}
  \nabla\cdot\big( (1+\chi(|\nabla U|))\nabla U\big) &= 0
  \hspace*{29.1mm}\text{in }\Omega_2, \\
  \Delta U_2 &= 0 
  \hspace*{29.1mm}\text{in }\R^3\backslash\overline\Omega_2, \\
  \label{eq:u:omega2:jumpu}
  U_2^{\rm ext} - U^{\rm int} &= -U_1^{\rm int} - U_{\text{app}}^{\rm int}
  \hspace*{10.5mm}\text{on }\Gamma_2, \\
  \label{eq:u:omega2:jumpdu}
  \revision{\big(\nabla U_2^{\rm ext} - (1+\chi(|\nabla U^{\rm int}|))\nabla 
  U^{\rm int}\big)\cdot\normal_2} &= (\H_1^{\revision{\rm int}} + \Hext^{\revision{\rm int}})\cdot\revision{\normal_2}
  \hspace*{-0.6mm}\quad\text{on }\Gamma_2, \\
  U_2(x) &= \OO(1/|x|) 
  \hspace*{18.0mm}\text{as }|x|\to\infty,
\end{align}
\end{subequations}
where~\eqref{eq:u:omega2:interior} follows
from~\eqref{eq:maxwell}--\eqref{eq:divHj} and~\eqref{eq:material}.
\revision{The transmission condition~\eqref{eq:u:omega2:jumpu} follows from the continuity of $U_2$ on $\Gamma_2$ and $U = U_1+U_2+U_{\text{app}}$ in
$\Omega_2$. 
To see~\eqref{eq:u:omega2:jumpdu}, we stress that~\eqref{eq:maxwell} implies $(\BB^{\rm ext}-\BB^{\rm int})\cdot\normal_2=0$ on $\Gamma_2$.
Putting~\eqref{eq:flux}--\eqref{eq:splitTotalMagField} into this condition and using $\H = -\nabla U$ in $\Omega_2$ as well
as~\eqref{eq:material} gives us
\begin{align*}
  (\H_1^{\rm ext} + \H_2^{\rm ext} + \Hext^{\rm ext} - (1+\chi(|\nabla U^{\rm int}|))\nabla U^{\rm int}) \cdot \normal_2=0 
  \quad\text{on }\Gamma_2.
\end{align*}
Moreover, from~\eqref{eq:Hext:div} and~\eqref{eq:divHj} we infer $(\Hext^{\rm ext}-\Hext^{\rm int})\cdot\normal_2 = 0 = (\H_1^{\rm
ext}-\H_1^{\rm int})\cdot\normal_2$ on $\Gamma_2$.
Together with $\H_2 = -\nabla U_2$, the transmission condition~\eqref{eq:u:omega2:jumpdu} follows.}
\begin{remark}
In case of a linear material law $\chi(|\H|) = \chi \in \R_{>0}$ in~\eqref{eq:material}, the transmission problem~\eqref{eq:u:omega2} simplifies to $(1+\chi)\Delta U_2 = 0$ in $\Omega_2$, $U_2^{\rm ext}-U_2^{\rm int} = 0$ on $\Gamma_2$, and $\revision{\big(\nabla U_2^{\rm ext} - (1+\chi)\nabla U_2^{\rm int}\big)\cdot\normal_2} = (\H_1^{\rm int} + \Hext^{\rm int}) \cdot \revision{\normal_2}$ on $\Gamma_2$ in~\eqref{eq:u:omega2:interior},~\eqref{eq:u:omega2:jumpu}, and~\eqref{eq:u:omega2:jumpdu}, respectively.
In particular, the Neumann problem~\eqref{eq:uext} does not have to be solved.
Moreover, we do not have to assume that $\Omega_2$ is simply connected.
\end{remark}
\subsection{Landau-Lifshitz-Gilbert equation}
\label{subsection:llg}
Let $\alpha> 0$ denote a dimensionless empiric damping parameter, called Gilbert damping constant, and let the magnetization of the ferromagnetic body $\Omega_1$ be characterized by the vector valued function
\begin{equation*}
\M_1:\, (0,T)\times\Omega_1 \rightarrow \set{\x\in\R^3}{|\x|=M_s},
\end{equation*}
where the constant $M_s>0$ refers to the saturation magnetization $[A/m]$.
Then, the Landau-Lifshitz-Gilbert equation reads
\begin{subequations}\label{eq:llg:physics}
\begin{align}\label{eq:llg_phy}
\frac{\partial\M_1}{\partial t}
= -\frac{\gamma_0}{1+\alpha^2} \M_1\times\Heff
- \frac{\alpha\gamma_0}{(1+\alpha^2)M_s} \M_1\times(\M_1\times\Heff),
\end{align}
supplemented by initial and Neumann boundary conditions
\begin{align}
 \M_1(0) &= \M^0
 \quad\text{in }\Omega_1,\\
 \partial_\normal\M_1 &= 0
 \hspace*{7.5mm}\text{on }(0,\revision{T})\times\partial\Omega_1.
\end{align}
\end{subequations}
Here, \revision{$\gamma_0 = 2,210173\cdot10^5$ $m/(As)$} denotes the gyromagnetic ratio
and $\M^0:\Omega_1\to\R^3$ with $|\M^0| = M_s$ in $\Omega_1$ is a given initial
magnetization.
The effective field $\Heff$ in $[A/m]$ depends on $\M_1$ and the magnetic field strength
$\H$, and is given as the negative \revision{first} variation of the Gibbs free energy
\begin{equation*}
\revision{\mu_0}\,\Heff = -\frac{\delta E(\M_1)}{\delta\M_1}.
\end{equation*}
In this work, the energy $E(\cdot)$ consists of exchange energy, anisotropy
energy as well as magnetostatic energy
\begin{equation*}
 E(\M_1) = \frac{A}{M_s^2}\,\int_{\Omega_1} |\nabla\M_1|^2
 + K\,\int_{\Omega_1}\phi(\M_1/M_s)
 - \mu_0\int_{\Omega_1}\H\cdot\M_1.
\end{equation*}
The exchange constant $A>0$ $[J/m]$ and anisotropy constant $K>0$ $[J/m^3]$ depend on the ferromagnetic material.
Moreover, $\phi$ refers to the crystalline anisotropy density.
The effective field is thus given by
\begin{equation*}
\Heff = \frac{2A}{\mu_0 M_s^2}\Delta \M_1 - \frac{K}{\mu_0\revision{M_s}}D\phi(\M_1\revision{/M_s}) + \H.
\end{equation*}
Note that the microscopic LLG equation and the macroscopic Maxwell equations are coupled through the magnetic field strength $\H$ and hence through the effective field $\Heff$.
Altogether, we will thus solve the multiscale problem by solving LLG on $\Omega_1$ and incorporating the effects of $\Omega_2$ via this coupling.
\section{General LLG equation}
\label{section:general}
\noindent In this section, we consider the non-dimensional form of LLG with a quite general effective field $\heff$ which covers the multiscale problem from the previous section.
We recall some equivalent formulations of LLG and then state our notion of a weak solution, which has been introduced by \textsc{Alouges} \& \textsc{Soyeur}, see Ref.~\refcite{as}, for the small-particle limit $\heff =\Delta \m$ and which is now extended to the present situation. We then formulate a linear-implicit
time integrator in the spirit of Refs.~\refcite{alouges08}, \refcite{alouges11}, \refcite{goldenits}, \refcite{mathmod2012}, \refcite{gamm2011}. 
\subsection{Non-dimensional form of LLG}
\noindent We perform the substitution $t' = \gamma_0M_st$ with $t'$ being the so-called 
(non-dimensional) reduced time, and set $T' = \gamma_0M_sT$ the scaled final time.
Moreover, we rescale the spatial variable $x'=x/L$ with $L$ being some characteristic
length of the problem $[m]$, e.g., the intrinsic length scale
$L = \sqrt{2A/(\mu_0M_s^2)}$. However, to simplify our notation, we stick with $t,T,x,\Omega_j$ instead of $t',T',x/L,\Omega_j/L$, respectively,
and abbreviate the space-time cylinder $\Omega_t = [0,t]\times\Omega_1$
for all $0\le t\le T$.
We set $\m := \M_1/M_s$, $\m^0 := \M^0/M_s$, $\heff:=\Heff/M_s$. With these
notations, the (sought) magnetization $\m:\Omega_T\to\set{x\in\R^3}{|x|=1}$
solves the non-dimensional form of LLG
\begin{subequations}\label{eq:llg}
\begin{align}\label{eq:llg1}
 \revision{\partial_t\m} 
 = - \frac{1}{1+\alpha^2}\,\m\times\heff
 -\frac{\alpha}{1+\alpha^2}\,\m\times(\m\times\heff) 
 \quad\revision{\text{in }\Omega_T},
\end{align}
supplemented by initial and \revision{Neumann} boundary conditions
\begin{align}\label{eq:llg:bc1}
 \m(0) &= \m^0\quad\mbox{in }\Omega_1,\\
 \label{eq:llg:bc2}
 \partial_\normal\m &=0 \quad\quad\mbox{in }(0,T)\times\partial\Omega_1.
\end{align}
\end{subequations}
The \revision{non-dimensional} effective field reads
\begin{align*}
 \heff = 
 \frac{2A}{\mu_0 M_s^2\revision{L^2}}\,\Delta\m
 - \frac{K}{\mu_0\revision{M_s^2}}\,D\phi(\m)
 + \hext - \nabla u_1 -\nabla u_2,
\end{align*}
where $u_1$ \revision{solves}~\eqref{eq:u1} with $\M_1$ being replaced by $\m$
and where $u_2$ \revision{solves}~\eqref{eq:u:omega2} with, e.g., $\Hext$ replaced by $\hext$,
$\H_1$ replaced by $-\nabla u_1$, etc. For the non-linearity $\chi$, we introduce some $\widetilde \chi$ in the non-dimensional formulation. Details are elaborated in Section~\ref{sec:multi}.
\begin{remark}
Note that~\eqref{eq:llg1} implies $0 = \m \cdot \partial_t \m = \partial_t |\m|^2/2$, i.e., the time derivative $\partial_t \m$ belongs to the tangent space 
of $\m$. \revision{In particular,} the modulus constraint $|\m| = 1$ in $\Omega_T$
also follows from the PDE formulation~\eqref{eq:llg1} \revision{and $|\m^0|=1$ in $\Omega_1$}.
\end{remark}
\subsection{Notation and function spaces involved}
\noindent In this brief section, we collect the necessary notation as well as the relevant function spaces that will be used \revision{throughout}.
By $L^2$, we denote the usual Lebesgue space of square integrable functions and by $H^1$
the Sobolev space of functions in $L^2$ that additionally admit a weak \revision{gradient} in $L^2$. For vector fields and corresponding spaces, we use bold symbols,
e.g., for $\f \in \L^2(\Omega_1)$, we write
\begin{align*}
\norm{\f}{\L^2(\Omega_1)}^2 = \sum_{i=1}^3 \norm{f_i}{L^2(\Omega_1)}^2.
\end{align*}
For the space-time cylinder $\revision{\Omega_T=[0,T]\times\Omega_1}$, we consider the 
\revision{function} spaces $L^2(\L^2):=L^2\big([0,\revision{T}], \L^2(\Omega_1)\big) = \L^2(\revision{\Omega_T})$, $L^2(\H^1):=L^2\big([0,\revision{T}], \H^1(\Omega_1)\big)$, and $\H^1(\revision{\Omega_T})$ which are associated with the norms
\begin{align*}
\norm{\f}{L^2(\L^2)}^2 &:=\norm{\f}{\L^2(\revision{\Omega_T})}^2 = \int_0^\revision{T} \norm{\f(t)}{\L^2(\Omega_1)}^2\, dt,\\
\norm{\f}{L^2(\H^1)}^2 &:=\norm{\f}{L^2([0,\revision{T}], \H^1(\Omega_1))}^2 = \int_0^{\revision{T}} \norm{\f(t)}{\L^2(\Omega_1)}^2 + \norm{\nabla \f(t)}{\L^2(\Omega_1)}^2\, dt,\\
\norm{\f}{\H^1(\revision{\Omega_T)}}^2 &= \int_0^{\revision{T}} \norm{\f(t)}{\L^2(\Omega_1)}^2 + \norm{\nabla \f(t)}{\L^2(\Omega_1)}^2 + \norm{\partial_t \f(t)}{\L^2(\Omega_1)}^2\, dt,
\end{align*}
respectively. 
Finally, for appropriate sets $\Sigma$, we denote by $\dual{\cdot}{\cdot}_\Sigma$ the scalar product of $\L^2(\Sigma)$.
The Euclidean scalar product of vectors $\boldsymbol{x},\boldsymbol{y} \in \R^3$ is denoted by $\boldsymbol{x} \cdot \boldsymbol{y}$.
In proofs, we use the symbol $\lesssim$ to abbreviate $\le$ up to some (hidden) multiplicative constant which is clear from the context and independent of the discretization parameters $h$ and $k$.
\subsection{Equivalent formulations of LLG and weak solution}
The dimensionless formulation of LLG that is usually referred to, has already been stated in~\eqref{eq:llg}.
Supplemented by the same initial and boundary conditions~\eqref{eq:llg:bc1}--\eqref{eq:llg:bc2}, the equation can also equivalently be stated \revision{as}
\begin{align}\label{eq:form:alg}
\alpha\partial_t\m + \m\times\partial_t\m
= \heff - \left(\m\cdot\heff\right)\m
\end{align}
\revision{or}
\begin{align}\label{eq:weaksol}
\partial_t\m - \alpha\m\times\partial_t\m
= \heff \times \m.
\end{align}
In this work, \eqref{eq:form:alg} is exploited for the construction of our
numerical scheme. For the notion of a weak solution,
we use the so-called Gilbert formulation~\eqref{eq:weaksol}. A rigorous proof for the equivalence of
the above equations can be found, e.g., in Ref.~\refcite{goldenits}, Section~1.2.
\par As far as numerical analysis is concerned, our integrator extends \revision{the one} of
Ref.~\refcite{alouges08} from the small-particle limit with exchange energy only, to the case under
consideration. Independently, the preceding works of Refs.~\refcite{alouges11}, \refcite{goldenits} generalized the
approach of Ref.~\refcite{alouges08} to an effective field, which consists of exchange energy, stray field energy, uniaxial anisotropy, and exterior energy, where only the first term is dealt with implicitly, whereas the remaining lower-order terms are treated explicitly. In this work, we extend this approach to certain non-linear contributions of the effective field. For this purpose, 
\revision{we introduce a general contribution 
$\operator:\H^1(\Omega_1)\times Y\to \L^2(\Omega_1)$ for some suitable Banach 
space $Y$, see Section~\ref{section:heff} for examples.}
We now write $\heff$ in the form
\begin{subequations}\label{se:multiscale}
\begin{align}\label{eq:field}
 \heff = \Cexchange\Delta\m - \operator(\m,\zeta) + \hext,
\end{align}
where $\zeta \in Y$, the exchange contribution and the exterior field $\hext$ are explicitly given, while the stray field contribution, the material anisotropy, and the induced field
from the macroscopic part are concluded in the operator \revision{$\operator$}. Our analysis thus particularly includes the case
\begin{align}\label{eq:pi}
 \revision{\operator\big(\m, \zeta\big)}
 := \nabla u_1 + \Canisotropy\,\revision{D\phi(\m)} +\nabla u_2,
\end{align}
but also holds true for general contributions $\operator$,
which only act on the spatial variable, as long as they fulfil \revision{the properties \eqref{assumption:chi:bounded}--\eqref{assumption:chi:convergence}} below. In~\eqref{eq:field}--\eqref{eq:pi}, the constants are given by
\begin{align}\label{eq:constants}
\Cexchange := \frac{2A}{\mu_0 M_s^2\revision{L^2}} \quad \text{ resp. } \quad \Canisotropy := \frac{K}{\mu_0 \revision{M_s}}.
\end{align}
\end{subequations}
\revision{\begin{remark}
For the multiscale formulation~\eqref{se:multiscale}, we employ 
\revision{$Y=\L^2(\Omega_2)$} 
and $\zeta=\hext$, since this data is required 
in~\eqref{eq:uext}--\eqref{eq:u:omega2}. Details are given in Section~\ref{sec:multi} below.
For the classical contributions like anisotropy field and stray field,
the operator $\operator$ is independent of $\zeta$ and depends only on $\m$.
\end{remark}}
\par With \revision{these preparations}, our \revision{definition} of a weak solution reads as follows:
\begin{definition}\label{def:weaksol}
\revision{Let $\hext\in\L^2(\Omega_1)$, $\zeta \in Y$ and $\m^0\in\H^1(\Omega_1)$ with $|m|=1$ in $\Omega_1$.}
A function $\m$ is called a \emph{weak solution} to LLG in $\revision{\Omega_T}$, if
\begin{itemize}
\item[(i)] $\m \in \H^1(\revision{\Omega_T})$ with $|\m| = 1$ in $\revision{\Omega_T}$ and $\m(0)=\m^0$ in the sense of traces;
\item[(ii)] for all $\pphi \in C^\infty(\revision{\overline\Omega_T})$, we have
\begin{align}\label{eq:weaksol_def}
&\dual{\partial_t\m}{\pphi}_{\Omega_T}
-\alpha\,\dual{\m \times \revision{\partial_t\m}}{\pphi}_{\revision{\Omega_T}}\\
&= -\Cexchange\,\dual{\nabla \m \times \m}{\nabla \pphi}_{\revision{\Omega_T}}
- \dual{\operator(\m, \zeta)\times \m}{\pphi}_{\revision{\Omega_T}}
+ \dual{\hext\times \m}{\pphi}_{\Omega_T};\nonumber
 \end{align}
\item[(iii)] for almost all $t \in (0, \revision{T})$, we have
\begin{align}\label{eq:weaksol:energy}
\norm{\nabla \m (t)}{L^2(\Omega_1)}^2 + \norm{\partial_t\m}{L^2(\Omega_t)}^2 \le C,
\end{align}
for some constant $C>0$ which depends only on $\m^0$ and $\hext$.
\end{itemize}
\end{definition}
The existence (and non-uniqueness) of weak solutions has first been
shown in Ref.~\refcite{as} for the small particle limit, where \revision{$\operator$}
and $\hext$ are omitted. We stress, however, that our convergence proof is
constructive in the sense that the analysis does not only show convergence
towards, but also existence of weak solutions without any assumptions on the smoothness of the quantities involved.
\begin{remark}
Under certain assumptions on \revision{$\operator$},
the energy estimate~\eqref{eq:weaksol:energy} can be improved. We refer to Proposition~\ref{lem:energy:improved} in the appendix.
\end{remark}
\subsection{Linear-implicit integrator}
We discretize the magnetization $\m$ and its time derivative $\v = \partial_t\m$ in space by lowest-order Courant finite elements
\begin{equation*}
 \VV_h := 
 \set{\n_h:\overline\Omega_1\to\R^3\text{ continuous}}{\n_h|_T\text{ affine for all }T\in \TT_h^{\Omega_1}},
\end{equation*}
where $\TT_h^{\Omega_1}$ is a \revision{quasi-uniform and} conforming triangulation of $\Omega_1$ into 
tetrahedra $T\in\TT_h^{\Omega_1}$ with mesh-size \revision{$h\simeq{\rm diam}(T)$}. Let $\NN_h$ denote the set of nodes of $\TT_h^{\Omega_1}$. For fixed time $t_j$, the discrete magnetization is sought in the set
\begin{align*}
 \m(t_j) \approx \m_h^j \in
 \MM_h := \set{\n_h\in\VV_h}{|\n_h(z)|=1\text{ for all }z\in\NN_h},
\end{align*}
whereas the discrete time derivative is sought in the discrete tangent space
\begin{align*}
 \v(t_j) \approx \v_h^j \in
 \KK_{\m_h^j} := \set{\n_h\in\VV_h}{\n_h(z)\cdot\m_h^j(z)=0\text{ for all }z\in\NN_h}.
\end{align*}
For the time discretization, we impose a uniform partition $\II_k$ of the time interval $[0,T]$ with time step-size $k=T/N$ and time steps $t_j = jk$, $j=0,\dots,N$.
\par Let $\operator_h$ be a numerical realization \revision{of $\operator$}
which maps $\m(t_j)\approx\m_h^j\in\MM_h$ and $\zeta(t_j) \approx \zeta_h^j\in Y$ to some
$\operator_h(\m_h^j, \zeta_h^j)\in\L^2(\Omega_1)$.
Finally, let \revision{$\hext_h^j\in\L^2(\Omega_1)$} be an approximation of $\hext(t_j)$ specified below. Then, our numerical time integrator reads as follows:
\begin{algorithm}\label{algorithm}
Input: Initial datum $\m_h^0\in\MM_h$, parameters $\alpha>0$ and $0 \leq\theta\leq 1$, data $\left\{\zeta_h^i\right\}_{i=0,\dots,N-1}$.
Then, for all $i=0,\dots,N-1$ iterate:
\begin{itemize}
\item[(i)] Compute $\v_h^i\in\KK_{\m_h^i}$ such that for all $\ppsi_h\in\KK_{\m_h^i}$, it holds
\begin{align}\label{eq:alg}
 &\alpha\dual{\v_h^i}{\ppsi_h}_{\Omega_1}
 + \Cexchange k \theta \,\dual{\nabla\v_h^i}{\nabla\ppsi_h}_{\Omega_1}
 + \dual{\m_h^i\times\v_h^i}{\ppsi_h}_{\Omega_1}
 \\&\quad
 = -\Cexchange \dual{\nabla\m_h^i}{\nabla \ppsi_h}_{\Omega_1}
 - \dual{\operator_h(\m_h^i,\zeta_h^i)}{\ppsi_h}_{\Omega_1}
 + \dual{\hext_h^i}{\ppsi_h}_{\Omega_1}. \nonumber
\end{align}
\item[(ii)] Define $\m_h^{i+1}\in\MM_h$ by
$\m_h^{i+1}(z) = \displaystyle\frac{\m_h^i(z) + k \v_h^i(z)}{|\m_h^i(z) + k \v_h^i(z)|}$ for all nodes $z\in\NN_h$.
\end{itemize}
Output: Discrete time derivatives $\v_h^i$ and magnetizations $\m_h^{i+1}$, for $i=0,\dots,N-1$.
\end{algorithm}
\revision{The input as well as the output of Algorithm~\ref{algorithm} consists of discrete-in-time values $\gamma_h^i$,
e.g., $\gamma_h^i\in\{\m_h^i,\v_h^i\}\subseteq\VV_h$.
By~\eqref{eq:timeapprox} we define continuous-in-time interpretations, where we consider continuous and 
piecewise affine in time (denoted by $\SS^1$) resp.\ piecewise constant in time 
(denoted by $\PP^0$): For $t_i \le t < t_{i+1}$,
$\gamma_{hk}\in\SS^1(\II_k;\VV_h)\subset\H^1(\revision{\Omega_T})$ 
and  $\gamma_{hk}^-\in\PP^0(\II_k; \VV_h)\subset L^2(\H^1)$ are defined by
\begin{subequations} \label{eq:timeapprox}
\begin{align}
\gamma_{hk}(t) &:= \frac{t-ik}{k}\,\gamma_h^{i+1} + \frac{(i+1)k-t}{k}\,\gamma_h^i
\label{eq:mhk}\\
\gamma_{hk}^-(t) &:= \gamma_h^i
\label{eq:mhk-}.
\end{align}
\end{subequations}
We note that $\partial_t\gamma_{hk}=(\gamma_h^{i+1}-\gamma_h^i)/k$.
The same notation is used for $\hext_{hk}^-\in\PP^0(\II_k;\L^2(\Omega))$ and 
$\zeta_{hk}^-\in\PP^0(\II_k;Y)$.}
\revision{\begin{lemma}
Algorithm~\ref{algorithm} is well-defined, and it holds $\norm{\m_{hk}}{\L^\infty(\revision{\Omega_T})} = \norm{\m_{hk}^-}{\L^\infty(\revision{\Omega_T})} = 1$.
\end{lemma}}
\begin{proof}
Problem~\eqref{eq:alg} is a linear problem on a finite dimensional space.
Therefore, existence and uniqueness of $\v_h^i\in\KK_{\m_h^i}$ follow from the fact that the corresponding bilinear form is positive definite.
By definition of the discrete tangent space \revision{$\KK_{\m_h^i}$}, it holds
$|\m_h^i + k\v_h^i|^2 = 1 + k^2\,|\v_h^i|^2\ge1$ nodewise.
Therefore, \revision{Step~(ii) in Algorithm~\ref{algorithm}} is well-defined. 
By use of barycentric coordinates, an elementary calculation finally proves the pointwise estimates
$|\m_{hk}^-|\le1$ as well as $|\m_{hk}| \le 1$, see, e.g., Ref.~\refcite{alouges08}.
\end{proof}
By definition of $\m_h^{i+1}$ in Step~(ii) of Algorithm~\ref{algorithm}, the
following two auxiliary results follow from elementary geometric considerations
(see Refs.~\refcite{alouges08}, \refcite{alouges11}, \refcite{goldenits}).
\revision{\begin{lemma}\label{lemma1:aux}
For all $i=0,\dots,N-1$, it holds nodewise $|\m_h^{i+1}-\m_h^i| \le k\,|\v_h^i|$.\hfill\qed
\end{lemma}}
\revision{\begin{lemma}\label{lemma2:aux}
For all $i=0,\dots,N-1$, it holds nodewise $|\m_h^{i+1}-\m_h^i-k\v_h^i| \le \frac12\,k^2\,|\v_h^i|^2$.\hfill\qed
\end{lemma}}
These nodal estimates shall be used together with the following elementary lemma which follows from standard scaling arguments.
\revision{\begin{lemma}\label{lemma3:aux}
For any discrete function ${\bf w}_h\in\VV_h$ and all $1\le p<\infty$, it holds
\begin{align*}
  \c{shape}^{-1}\,\norm{{\bf w}_h}{\L^p(\Omega)}^p
  \le h^3\sum_{z\in\NN_h}|{\bf w}_h(z)|^p
 \le \c{shape}\,\norm{{\bf w}_h}{\L^p(\Omega)}^p.
\end{align*}
The constant $\setc{shape}>0$ depends only on $p$ and the shape of the
elements in $\TT_h^{\Omega_1}$.\hfill\qed
\end{lemma}}
\subsection{Main theorem}
The following theorem is the main result of this work.
It states convergence of the numerical integrator (at least for a subsequence) towards a weak solution of the general LLG equation.
Afterwards, we will show that the operator $\operator$ and its discretization $\operator_h$ of the multiscale LLG equation satisfy the general assumptions posed.
In particular, the concrete problem is thus covered by the general approach.
\begin{theorem}\label{theorem}
\textbf{(a)}
Let \revision{$1/2<\theta\le1$} and suppose that the spatial meshes $\TT_h^{\Omega_1}$ are uniformly shape regular
and satisfy the angle condition
\begin{align}\label{assumption:mesh}
 \dual{\nabla\eta_i}{\nabla\eta_j}_{\Omega_1}
 \le0
 \quad\text{for \revision{all nodal hat} functions }
 \eta_i,\eta_j\in\SS^1(\TT_h^{\Omega_1})
 \text{ with }i\neq j.
\end{align}
We suppose that
\begin{align}\label{assumption:f}
 \hext_{hk}^- \weakto \hext\text{ weakly in }\L^2(\revision{\Omega_T})
\end{align}
as well as
\begin{align}\label{assumption:m0}
 \revision{\m_h^0 \weakto \m^0\text{ weakly in }\H^1(\Omega_1).}
\end{align}
Moreover, we suppose that the spatial discretization $\revision{\operator_h}$
of \revision{$\operator$} satisfies
\begin{align}\label{assumption:chi:bounded}
 \norm{\operator_h(\n, y)}{\L^2(\Omega_1)}
 \le \c{bounded}\,\revision{(1+\norm{\nabla\n}{L^2(\Omega_1)})}
\end{align}
for all $h,k>0$ and all $\n \in \revision{\H^1(\Omega_1)}$ with $|\n| \le 1$ and \revision{all} $y \in Y$ with $\norm{y}{Y} \le \c{boundedy}$ for some $y$-independent constant $\c{boundedy}>0$.
Here, $\setc{bounded}>0$ denotes a constant that is independent of $h,k, \n,$ and $y$, but may depend on $\c{boundedy}$ and $\Omega_1$. We further assume $\norm{\zeta_h^j}{Y} \le \setc{boundedy}$ for all $j = 1, \hdots, N$. Under these assumptions, \revision{Algorithm~\ref{algorithm} yields} \revision{strong $\L^2(\Omega_T)$-convergence
of some subsequence of $\m_{hk}^-$ as well as weak $\H^1(\Omega_T)$-convergence 
of some subsequence of $\m_{hk}$ towards the same limit $\m\in \H^1(\Omega_T)$
which additionally satisfies $\m\in L^\infty(\H^1)$ with $|\m|=1$ in $\Omega_T$.}\\
\par \noindent
\textbf{(b)}
In addition to the above, we suppose
\begin{align}\label{assumption:chi:convergence}
 \operator_h(\m_{hk}^-, \zeta_{hk}^-) \weakto \operator(\m, \zeta)
 \quad\text{weakly in }\revision{\L^2(\Omega_T)} \text{ \revision{for some subsequence}}.
\end{align}
Then, the \revision{limit $\m\in\H^1(\Omega_T)$ from (a) is 
a weak solution} of general LLG
\revision{in the sense of Definition~\ref{def:weaksol}.}
\end{theorem}
\begin{remark} \label{rem:f}
{\rm(i)} Suppose that the applied exterior field is continuous in time, i.e., $\hext\in C([0,T];\L^2(\Omega_1))$. Let $\hext_h^j = \hext(t_j)$ denote the evaluation of $\hext$ at time $t_j$. Then, assumption~\eqref{assumption:f} is satisfied since $\hext_{hk}^-\to\hext$ strongly in $L^\infty(\L^2)$.\\
{\rm(ii)} Suppose that the applied exterior field is continuous in space-time, i.e., $\hext\in C(\revision{\overline\Omega_T})$. Let $\hext_h^j$ denote the nodal interpolant of $\hext(t_j)\in C(\overline\Omega_1)$ in space. Then, assumption~\eqref{assumption:f} is satisfied since $\hext_{hk}^-\to\hext$ strongly in $\L^\infty(\revision{\Omega_T})$.\\
{\rm(iii)} Suppose $\zeta$ is continuous in time, i.e., \revision{$\zeta \in C([0,T], Y)$} and let $\zeta_h^j = \zeta(t_j)$ denote the evaluation of $\zeta$ at time $t_j$. Then, we have $\zeta_{hk}^- \to \zeta$ strongly in $L^\infty(Y)$ and $\norm{\zeta_h^j}{Y} \le \sup_{t \in [0,T]} \norm{\zeta(t)}{Y}$.
\end{remark}
\begin{remark}
The angle condition~\eqref{assumption:mesh} is a technical ingredient for 
the convergence analysis. It is automatically fulfilled for tetrahedral meshes 
with dihedral angles that are smaller than $\pi/2$. If the
condition is satisfied by the initial mesh $\TT_0$, \revision{it can be preserved
by the mesh-refinement strategy (see, e.g., Ref.~\refcite{verfuerth}, Section~4.1).}
%
\end{remark} 
\revision{The remainder of this section consists of the proof of
Theorem~\ref{theorem} which is roughly split into three steps:}
\begin{itemize}
\item[(i)] Boundedness of the discrete quantities and energies.
\item[(ii)] Existence of weakly convergent subsequences.
\item[(iii)] Identification of the limits with weak solutions of LLG.
\end{itemize}
\begin{lemma}\label{lem:energy}
\revision{For all $j=0,\dots,N$, the}
discrete quantities $\m_h^j$ and $\left\{\v_h^i\right\}_{i=0,\dots,j-1}$ satisfy
\begin{align}\label{eq:energy_discrete}
\begin{split}
\norm{\nabla \m_h^j}{\L^2(\Omega_1)}^2 &+ 
k\sum_{i=0}^{j-1}\norm{\v_h^i}{\L^2(\Omega_1)}^2 
+ (\theta - 1/2)k^2
\sum_{i=0}^{j-1} \norm{\nabla \v_h^i}{\L^2(\Omega_1)}^2 
\le\revision{\c{energy}}.
\end{split}
\end{align}
\revision{The constant $\setc{energy} > 0$ depends only on $\hext$, $\m^0$, 
and the final time $T$, but is independent of $h$ and $k$.}
\end{lemma}
\begin{proof}
In~\eqref{eq:alg}, we use the test function $\ppsi_h = \v_h^i \in \KK_{\m_h^i}$ 
and get
\begin{align*}
\alpha \norm{\v_h^i}{\L^2(\Omega_1)}^2 + \Cexchange \theta\,k \norm{\nabla \v_h^i}{\L^2(\Omega_1)}^2
= &-\Cexchange\dual{\nabla \m_h^i}{\nabla \v_h^i}_{\Omega_1}
+ \dual{\hext_h^i}{\v_h^i}_{\Omega_1} \\
&- \dual{\operator_h(\m_h^i,\zeta_h^i)}{\v_h^i}_{\Omega_1}.
\end{align*}
The angle condition~\eqref{assumption:mesh} ensures
$\norm{\nabla\m_h^{i+1}}{\L^2(\Omega_1)}^2 \le \norm{\nabla (\m_h^i + k \v_h^i)}{\L^2(\Omega_1)}^2$, see Refs. \refcite{alouges08}, \refcite{alouges11}, \refcite{goldenits}.
We thus get
\begin{align}
\frac12\norm{\nabla \m_h^{i+1}}{\L^2(\Omega_1)}^2 &\le \frac12 \norm{\nabla \m_h^i}{\L^2(\Omega_1)}^2 +
k\dual{\nabla\m_h^i}{\nabla\v_h^i}_{\Omega_1}
+ \frac{k^2}{2} \norm{\nabla \v_h^i}{\L^2(\Omega_1)}^2 \nonumber\\
& \leq \frac12 \norm{\nabla \m_h^i}{\L^2(\Omega_1)}^2 - (\theta - 1/2)k^2\norm{\nabla \v_h^i}{\L^2(\Omega_1)}^2 \label{eq:nabla_m_bounded} \\
&\quad - \frac{\alpha\, k}{\Cexchange}\norm{\v_h^i}{\L^2(\Omega_1)}^2
+ \frac{k}{\Cexchange}\dual{\hext_h^i}{\v_h^i}_{\Omega_1}
- \frac{k}{\Cexchange}\dual{\operator_h(\m_h^i,\zeta_h^i)}{\v_h^i}_{\Omega_1}. \nonumber
\end{align}
Next, we sum up over $i = 0, \hdots, j-1$ to see
\begin{align*}
\frac12 \norm{\nabla \m_h^j}{\L^2(\Omega_1)}^2 \le &\frac12\norm{\nabla
\m_h^0}{\L^2(\Omega_1)}^2 - (\theta - 1/2)k^2 \sum_{i=0}^{j-1}
\norm{\nabla \v_h^i}{\L^2(\Omega_1)}^2 \\
&- \frac{\alpha k}{\Cexchange}\sum_{i=0}^{j-1}\norm{\v_h^i}{\L^2(\Omega_1)}^2 
+ \frac{k}{\Cexchange}\sum_{i=0}^{j-1}\big(\dual{\hext_h^i}{\v_h^i}_{\Omega_1}
- \dual{\operator_h(\m_h^i, \zeta_h^i)}{\v_h^i}_{\Omega_1}\big).
\end{align*}
Using the inequalities of Young and H\"older, this can be further estimated by
\begin{align*}
\frac12 &\norm{\nabla \m_h^j}{\L^2(\Omega_1)}^2 
+ \frac{k}{\Cexchange}(\alpha - \eps)\sum_{i=0}^{j-1}\norm{\v_h^i}{\L^2(\Omega_1)}^2
+ (\theta - 1/2)k^2 \sum_{i=0}^{j-1}\norm{\nabla \v_h^i}{\L^2(\Omega_1)}^2
\\ 
&\le \frac12\norm{\nabla\m_h^0}{\L^2(\Omega_1)}^2 
 + \frac{k}{\revision{2} \Cexchange \eps}\sum_{i=0}^{j-1}\big(\norm{\hext_h^i}{\L^2(\Omega_1)}^2 + \norm{\operator_h(\m_h^i, \zeta_h^i)}{\L^2(\Omega_1)}^2\big)
\end{align*}
for any $\eps > 0$. 
\revision{With the boundedness~\revision{\eqref{assumption:chi:bounded}}
of $\operator_h$, the last sum is estimated by
\begin{align*}
 k \sum_{i=0}^{j-1}(\norm{\hext_h^i}{\L^2(\Omega_1)}^2 \!+\! \norm{\operator_h(\m_h^i, \zeta_h^i)}{\L^2(\Omega_1)}^2)
 &\lesssim \norm{\hext_{hk}^-}{\L^2(\Omega_T)}^2
 \!+\! k\sum_{i=0}^{j-1}(1+\norm{\nabla\m_h^i}{\L^2(\Omega_1)}^2)
 \\
 &\lesssim \norm{\hext_{hk}^-}{\L^2(\Omega_T)}^2
 + T + k\,\sum_{i=0}^{j-1}\norm{\nabla\m_h^i}{\L^2(\Omega_1)}^2.
\end{align*}
Choosing $\eps<\alpha$, 
we altogether obtain
\begin{align*}
&\norm{\nabla \m_h^j}{\L^2(\Omega_1)}^2 
+ k\,\sum_{i=0}^{j-1}\norm{\v_h^i}{\L^2(\Omega_1)}^2
+ (\theta - 1/2)k^2 \sum_{i=0}^{j-1}\norm{\nabla \v_h^i}{\L^2(\Omega_1)}^2
\\ 
&\qquad\lesssim \norm{\hext_{hk}^-}{\L^2(\Omega_T)}^2
 + T + k\,\sum_{i=0}^{j-1}\norm{\nabla\m_h^i}{\L^2(\Omega_1)}^2.
\end{align*}
According to weak convergence~\eqref{assumption:f}--\eqref{assumption:m0},
there holds uniform boundedness $\norm{\hext_{hk}^-}{\L^2(\Omega_T)}^2 + 
\norm{\nabla\m_h^0}{\L^2(\Omega_1)}^2 \le C$. Consequently, the discrete
Gronwall lemma \revision{(see, e.g., Ref.~\refcite{thomee}, Lemma 10.5)} applies and concludes the proof.}
\end{proof}
\revision{As a consequence of the energy estimate~\eqref{eq:energy_discrete}, 
we obtain uniform boundedness of the discrete quantities.}
\revision{\begin{lemma}\label{lemma:dpr}
For $1/2\le\theta\le1$, it holds
\begin{align}
 \begin{split}\label{dpr:boundedness}
 &\norm{\m_{hk}^-}{L^\infty(\H^1)}
 + \norm{\m_{hk}}{L^\infty(\H^1)}
 + \norm{\partial_t\m_{hk}}{\L^2(\Omega_T)}\\
 &\qquad+ \norm{\v_{hk}^-}{\L^2(\Omega_T)}
 + \sqrt{(\theta-1/2)k}\,\norm{\nabla\v_{hk}^-}{\L^2(\Omega_T)}
 \le \c{dpr},
 \end{split}
\end{align}
where $\setc{dpr}>0$ does not depend on $h$ or $k$.
\end{lemma}}
\revision{\begin{proof}
Estimate~\eqref{eq:energy_discrete} reveals
\begin{align*}
 \max_{j=0,\dots,N}\norm{\nabla\m_h^j}{\L^2(\Omega_1)}^2
 + &
 \norm{\v_{hk}^-}{\L^2(\Omega_T)}^2
 + (\theta-1/2)k\,\norm{\nabla\v_{hk}^-}{\L^2(\Omega_T)}^2
 \lesssim 
 \c{energy}.
\end{align*}
Clearly, it holds 
\begin{align*}
 \norm{\nabla\m_{hk}}{L^\infty(\L^2)}^2
 + \norm{\nabla\m_{hk}^-}{L^\infty(\L^2)}^2 
 \lesssim 
 \max_{j=0,\dots,N}\norm{\nabla\m_h^j}{\L^2(\Omega_1)}^2.
\end{align*}
Together with $\norm{\m_{hk}}{\L^\infty(\Omega_T)}=1=\norm{\m_{hk}^-}{\L^\infty(\Omega_T)}$, this bounds the $L^\infty(\H^1)$-norms of $\m_{hk}$
and $\m_{hk}^-$.
For $t_j\le t<t_{j+1}$, Lemma~\ref{lemma1:aux} 
\revision{and Lemma~\ref{lemma3:aux}} prove
\begin{align*}
 \norm{\partial_t\m_{hk}(t)}{\L^2(\Omega_1)}^2
 = \norm{(\m_h^{j+1}-\m_h^{j})/k}{\L^2(\Omega_1)}^2
 \lesssim \norm{\v_h^j}{\L^2(\Omega_1)}^2,
\end{align*}
whence $\norm{\partial_t\m_{hk}}{\L^2(\Omega_T)}^2 
\lesssim \norm{\v_{hk}^-}{\L^2(\Omega_T)}^2$. This concludes the proof.
\end{proof}}%
\revision{Using~\eqref{dpr:boundedness}, we can extract weakly convergent
subsequences.}
\begin{lemma}\label{lem:subsequences}
There exist functions $\m \in \H^1(\Omega_T)$ and $\v \in \L^2(\Omega_T)$ such that
\begin{align*}
\begin{split}
 \m_{hk} &\rightharpoonup \m \text{ weakly in } \H^1(\Omega_T),\\
 \m_{hk}, \m_{hk}^- &\rightharpoonup \m \text{ weakly in } \revision{L^2(\H^1)}, \\
 \m_{hk}, \m_{hk}^- &\rightarrow \m \text{ strongly in } \L^2(\Omega_T),\\
 \v_{hk}^- &\rightharpoonup \v \text{ weakly in } \L^2(\Omega_T),
\end{split}
\end{align*}
as $(h,k) \rightarrow (0,0)$ independently of each other. Here, the 
\revision{convergences are to} be understood for one particular subsequence that is successively \revision{extracted}. 
\end{lemma}
\begin{proof}
\revision{Due to the uniform boundedness~\eqref{dpr:boundedness}, 
one may extract weakly convergent subsequences (with possibly different limits).}
It thus only remains to show, that the limits coincide, e.g.,
\begin{align*}
 \revision{\m_{hk}^-\rightharpoonup\m \text{ weakly in } \L^2(\Omega_T)\text{ and }  L^2(\H^1),}
\end{align*}
where \revision{$\m_{hk}\rightharpoonup\m$ weakly} in $\H^1(\Omega_T)$. 
Due to the Rellich compactness theorem, we have
$\m_{hk}\to\m$ strongly in $\L^2(\Omega_T)$.
We rewrite $\m_{hk}$ for $t_j\le t < t_{j+1}$ as
\begin{align*}
\m_{hk}\revision{(t)} = \m_h^j + \frac{t-t_j}{k}(\m_h^{j+1} - \m_h^j).
\end{align*}
\revision{Lemma~\ref{lemma1:aux} and Lemma~\ref{lemma3:aux} thus yield}
\begin{align*}
&\norm{\m_{hk} - \m_{hk}^-}{\L^2(\Omega_T)}^2 = \sum_{j=0}^{N-1}\int_{t_j}^{t_{j+1}} \norm{\m_h^j + \frac{t-t_j}{k}(\m_h^{j+1} - \m_h^j) - \m_h^j}{\L^2(\Omega_1)}^2\\
&\quad\le \sum_{j=0}^{N-1}\int_{t_j}^{t_{j+1}} k^2\norm[\Big]{\frac{\m_h^{j+1} - \m_h^j}{k}}{\L^2(\Omega_1)}^2 \lesssim 
k^3 \sum_{j=0}^{N-1}\norm{\v_h^j}{\L^2(\Omega_1)}^2 \to 0.
\end{align*}
This proves the result for $\L^2(\Omega_T)$. From the uniqueness of weak limits and the continuous inclusion $L^2(\H^1) \subseteq L^2(\Omega_T)$, we also conclude the result for $\revision{L^2(\H^1)}$.
\end{proof}

Next, we identify the limit function $\v$.

\begin{lemma}\label{lem:v}
It holds $\v = \revision{\partial_t\m}$.
\end{lemma}

\begin{proof}
\revision{For $t_j\le t<t_{j+1}$, Lemma~\ref{lemma2:aux} and
Lemma~\ref{lemma3:aux} prove
\begin{align*}
 \norm{\partial_t\m_{hk}(t) -\v_{hk}(t)}{\L^1(\Omega_1)}
 = \norm{(\m_h^{j+1}-\m_h^j)/k - \v_h^j}{\L^1(\Omega_1)}
 \lesssim k\,\norm{\v_h^j}{\L^2(\Omega_1)}^2.
\end{align*}
Integration in time yields}
\begin{align*}
\norm{\partial_t\m_{hk} -\v_{hk}}{\L^1(\Omega_T)} \lesssim k \norm{\v_{hk}}{\L^2(\Omega_T)}^2.
\end{align*}
Exploiting weak semi-continuity of $\norm{\cdot}{\L^1(\Omega_T)}$, \revision{we obtain}
\begin{align*}
\norm{\revision{\partial_t\m} - \v}{\L^1(\Omega_T)} \le \liminf_{\revision{(h,k)\to0}} \norm{\partial_t\m_{hk} -\v_{hk}}{\L^1(\Omega_T)} = 0
\end{align*}
and thus \revision{prove} the desired result.
\end{proof}

\revision{So far, we have only used the boundedness 
assumptions~\eqref{assumption:f}--\eqref{assumption:chi:bounded} and $\theta\ge1/2$. 
To conclude the proof of Theorem~\ref{theorem}~$(a)$, it remains to prove that 
$|\m|=1$ in $\Omega_T$ (Definition~\ref{def:weaksol} (i)).
We also note that bounded energy 
(Definition~\ref{def:weaksol} (iii)) is already a direct consequence of 
Lemma~\ref{lemma:dpr}.}

\medskip

\revision{\noindent{\bf Verification of Definition~\ref{def:weaksol} (i).}}
From
\begin{align*}
\norm{|\m| - 1}{\L^2(\Omega_T)} \le \norm{|\m| - |\m_{hk}^-|}{\L^2(\Omega_T)} + \norm{|\m_{hk}^-|-1}{\L^2(\Omega_T)}
\end{align*}
and
\begin{align*}
\norm{|\m_{hk}^-(t,\cdot)|-1}{\L^2(\Omega_1)} \le h \max_{j=0,\dots,N}\norm{\nabla \m_h^j}{\L^2(\Omega_1)},
\end{align*}
we deduce $|\m| = 1$ almost everywhere in $\Omega_T$. \revision{Together with
$\m_{hk}(0)=\m_h^0$, the equality 
$\m(0) = \m^0$ in the trace sense follows from the convergences 
$\m_h^0\weakto\m^0$ weakly in $\H^1(\Omega_1)$ as well as 
$\m_{hk}\weakto\m$ weakly in $\H^1(\Omega_T)$ (at least for a subsequence) 
and thus weak convergence of the traces.}
\hfill\qed\\

\revision{To prove Theorem~\ref{theorem} (b), it remains to show that 
the limit function $\m$ also satisfies Definition~\ref{def:weaksol} (ii).
This is done in the following and requires assumption~\eqref{assumption:chi:convergence} as well as $\theta>1/2$.}

\revision{\noindent{\bf Verification of Definition~\ref{def:weaksol} (ii).}}
Let $\pphi \in C^\infty(\revision{\overline\Omega_T})$ be arbitrary. We define test functions by \revision{$\ppsi_h:=\II_h(\m_{hk}^-\times \pphi)$,
where $\II_h:C(\overline\Omega)\to\VV_h$ denotes the nodal interpolation operator
which only acts on the spatial variable}.
\revision{Note that $\psi_h(t)\in\KK_{\m_h^j}$ for all $t_j\le t<t_{j+1}$.
Integration of~\eqref{eq:alg} in time thus gives}
\begin{align*}
\alpha\int_0^{T} &\dual{\v_{hk}^-}{\ppsi_h}_{\Omega_1}
+ \Cexchange k\theta \int_0^{T} \dual{\nabla\v_{hk}^-}{\nabla \ppsi_h}_{\Omega_1}
+ \int_0^{T} \dual{\m_{hk}^-\times\v_{hk}^-}{\ppsi_h}_{\Omega_1}\\
&= -\Cexchange \int_0^{T} \dual{\nabla\m_{hk}^-}{\nabla\ppsi_h}_{\Omega_1}
- \int_0^{T} \dual{\operator_h(\m_{hk}^-,\zeta_{hk}^-)}{\ppsi_h}_{\Omega_1}
+ \int_0^T \dual{\hext_{hk}^-}{\ppsi_h}_{\Omega_1}.
\end{align*}
Exploiting the approximation properties of $\II_h$ for $\ppsi=\m_{hk}^-\times\phi$, 
we get
\begin{align*}
&\int_0^T \dual{\alpha \v_{hk}^- + \m_{hk}^- \times \v_{hk}^-}{\m_{hk}^- \times \pphi}_{\Omega_1}
+ \Cexchange k\theta\int_0^{T} \dual{\nabla \v_{hk}^-}{\nabla(\m_{hk}^- \times \pphi)}_{\Omega_1}\\
&\quad+ \Cexchange\int_0^{T}\dual{\nabla \m_{hk}^-}{\nabla(\m_{hk}^- \times \pphi)}_{\Omega_1}
+ \int_0^{T} \dual{\operator_h(\m_{hk}^-,\zeta_{hk}^-)}{\m_{hk}^-\times \pphi}_{\Omega_1}\\
&\quad- \int_0^{T} \dual{\hext_{hk}^-}{\m_{hk}^-\times \pphi}_{\Omega_1}
 =\mathcal{O}(h).
\end{align*}
Next, we proceed as in Refs.~\refcite{alouges08}, \refcite{goldenits} to see that
\begin{align}\label{eq:nablav}\nonumber
&\int_0^{T} \dual{\alpha \v_{hk}^- + \m_{hk}^- \times \v_{hk}^-}{\m_{hk}^-\times\pphi}_{\Omega_1}
\longrightarrow \int_0^{T} \dual{\alpha\partial_t\m + \m \times \partial_t\m}{\m \times \pphi}_{\Omega_1},\\
&k\,\theta\int_0^{T} \dual{\nabla \v_{hk}^-}{\nabla(\m_{hk}^- \times \pphi)}_{\Omega_1}
\longrightarrow 0, \quad \text{ and }
\\
&\int_0^{T}\dual{\nabla \m_{hk}^-}{\nabla(\m_{hk}^- \times \pphi)}_{\Omega_1}
=\int_0^{T}\dual{\nabla \m_{hk}^-}{\m_{hk}^- \times \nabla\pphi}_{\Omega_1}
\longrightarrow \int_0^{T}\dual{\nabla\m}{\m \times\nabla\pphi}_{\Omega_1}.\nonumber
\end{align}
Here, we have used the boundedness of $\sqrt{k}\norm{\nabla \v_{hk}^-}{\L^2(\Omega_T)}$, which follows from~\revision{\eqref{dpr:boundedness}}
and \revision{$1/2<\theta\le1$}.
From the convergence $\m_{hk}^-\times \pphi \rightarrow \m \times \pphi$ strongly in $\L^2(\Omega_T)$ and the assumptions~\eqref{assumption:f} and~\eqref{assumption:chi:convergence} on $\hext_{hk}^-$ and $\operator_h(\m_{hk}^-, \zeta_{hk}^-)$, we conclude
\begin{align*}
\int_0^{T} \dual{\operator_h(\m_{hk}^-, \zeta_{hk}^-)}{\m_{hk}^-\times \pphi}_{\Omega_1}
&\longrightarrow \int_0^{T} \dual{\operator(\m, \zeta)}{\m \times \pphi}_{\Omega_1},
\quad \text{ and }\\
\int_0^{T} \dual{\hext_{hk}^-}{\m_{hk}^- \times \pphi}_{\Omega_1}
&\longrightarrow \int_0^{T} \dual{\hext}{\m \times \pphi}_{\Omega_1}.
\end{align*}
Altogether, we have now shown
\begin{align*}
&\alpha \int_0^{T} \dual{\partial_t\m}{\m \times \pphi}_{\Omega_1}
+ \int_0^{T} \dual{\m \times \partial_t\m}{\m \times \pphi}_{\Omega_1} = \\
&\quad-\Cexchange \int_0^{T} \dual{\nabla \m}{\nabla(\m \times \pphi)}_{\Omega_1}
-\int_0^{T} \dual{\operator(\m,\zeta)}{\m \times \pphi}_{\Omega_1}
+ \int_0^{T} \dual{\hext}{\m \times \pphi}_{\Omega_1}.
\end{align*}
Using the identity
$(\m \times \partial_t\m)\cdot(\m \times \pphi) = \revision{\partial_t\m} \cdot \pphi,$
we conclude~\eqref{eq:weaksol_def}.
\hfill\qed\\

\begin{remark}
Note that in case of the Crank-Nicholson-type scheme $(\theta = 1/2)$ one needs an additional bound for $\nabla\v_{hk}^-$ in~\eqref{eq:nablav}.
As in Refs.~\refcite{alouges08}, \refcite{alouges11}, \refcite{goldenits}, this can be obtained from an inverse estimate.
In this case, however, we end up with a (weak) coupling of $h$ and $k$, but still prove convergence as long as $k/h$ tends to $0$.
\end{remark}
\section{Effective field contributions for multiscale LLG equation}
\label{section:heff}
\noindent In this section, we give examples for contributions $\operator$ and corresponding discretizations $\operator_h$ which guarantee the assumptions~\eqref{assumption:chi:bounded}--\eqref{assumption:chi:convergence} of Theorem~\ref{theorem}.
In particular, we show that the contributions of our multiscale LLG model satisfy these assumptions.
\subsection{Pointwise operators and anisotropy energy contribution}\label{sec:anisotropy}
With $\B:=\set{x\in\R^3}{|x|\le1}$ the compact unit ball in $\R^3$, let $\phi:\B\to\R$ be a \revision{continuously differentiable} anisotropy density. Possible examples include the uniaxial density $\phi(x) = -\frac12\,(x\cdot\e)^2$ with a given easy axis \revision{$\e\in\R^3$ with $|\e|=1$} 
as well as the cubic density $\phi(x) = K_1(x_1^2x_2^2+x_2^2x_3^2) + K_2 x_1^2x_2^2x_3^2$ with certain constants $K_1,K_2\ge0$. 
The anisotropy contribution to the effective field reads
\begin{align*}
 \operator(\n, \zeta) = \operator(\n) = D\phi\circ\n
 \quad\text{for }\n\in\L^2(\Omega_1),
\end{align*}
and $\operator_h = \operator$. Note that in this case, we neglected a possible dependence on $\zeta$, i.e., formally $Y = \{0\}$ and $\zeta_{hk}^-$ denotes the constant zero sequence.
\begin{proposition}\label{lemma:anisotropy}
Suppose that $\boldsymbol\Phi\in \revision{C(\B)}$, e.g., $\boldsymbol\Phi(x)= D\phi(x)$, and $\operator_h(\n):=\operator(\n):=\boldsymbol\Phi\circ\n$. Then, the assumptions~\eqref{assumption:chi:bounded}--\eqref{assumption:chi:convergence} of Theorem~\ref{theorem} are satisfied.
\end{proposition}
\begin{proof}
Clearly,~\eqref{assumption:chi:bounded} holds with $\c{bounded} = \norm{\boldsymbol\Phi}{\revision{\L^\infty(\B)}}$. Part $(a)$ of Theorem~\ref{theorem} thus predicts convergence of a subsequence $\m_{hk}^- \to \m$ strongly in $\L^2(\Omega_T)$.
Now, choose sequences $h_\ell\to0$, $k_\ell\to0$ such that $\m_\ell:=\m_{h_\ell
k_\ell}^-$ converges strongly in $\L^2(\Omega_T)$ to $\m$. By extracting a
subsequence, we may in particular assume that $\m_\ell$ converges to $\m$ even
pointwise almost everywhere in $\Omega_T$. This implies
$\operator(\m_\ell)\to\operator(\m)$ pointwise almost everywhere in
$\Omega_T$. 
Moreover and because of~\eqref{assumption:chi:bounded}, $|\operator(\m)-\operator(\m_\ell)|\le2\c{bounded}$ is uniformly bounded in $\L^\infty(\Omega_T)$.
Finally, the Lebesgue dominated convergence theorem thus applies and proves even strong convergence of $\operator(\m_\ell)$ to $\operator(\m)$ in $\L^2(\Omega_T)$.
\end{proof}
\subsection{Notation and function spaces}
This section collects the notational and mathematical preliminaries needed for the discretization of the stray field (Section~\ref{section:fredkinkoehler}) as well as the multiscale contribution (Section~\ref{sec:multi}).
\subsubsection{Function spaces and trace operators}
\revision{By $\gamma_j^{\rm int}:H^1(\Omega_j)\to H^{1/2}(\Gamma_j)$, 
we denote the interior trace operator 
on $\Gamma_j=\partial\Omega_j$, i.e., $\gamma_j^{\rm int} = v|_{\Gamma_j}$ 
for functions
$v\in C(\overline\Omega_j)$. Likewise, $\gamma_j^{\rm ext}$ denotes the exterior
trace operator.}
Let $H_*^1(\Omega_j):=\set{v\in H^1(\Omega_j)}{\dual{v}{1}_{\Omega_j} = 0}$ and
$H_0^1(\Omega_j):=\set{v\in H^1(\Omega_j)}{\gamma_j^{\rm int}v = 0}$.
\par \revision{With the unit normal vector $\normal_j$ 
on $\Gamma_j$ which points from $\Omega_j$ to \revision{$\R^3\backslash\overline\Omega_j$}, we 
denote by $\delta_j^{\rm int}$ resp.\ $\delta_j^{\rm ext}$ the interior resp.\
exterior normal derivative with respect to $\normal_j$. These are formally defined by the first Green's formula for functions $v\in H^1(\Omega)$ with 
$\Delta v\in L^2(\Omega)$.
For smooth functions, it holds $\delta_j^{\rm int}v = \nabla v\cdot\normal_j = \delta_j^{\rm ext}v$.}
\par Let $\TT_h^{\Omega_j}$ denote a \revision{quasi-uniform and} conforming triangulation of $\Omega_j$ into tetrahedra $T\in\TT_h^{\Omega_j}$ with mesh-size $h\simeq{\rm diam}(T)$.
We denote by $\SS^1(\TT_h^{\Omega_j})$ the space of piecewise affine and globally continuous functions on $\TT_h^{\Omega_j}$.
We define the discrete function spaces 
$\SS_*^1(\TT_h^{\Omega_j}) = H_*^1(\Omega_j) \cap \SS^1(\TT_h^{\Omega_j})$ 
resp.\ $\SS_0^1(\TT_h^{\Omega_j}) = H_0^1(\Omega_j) \cap \SS^1(\TT_h^{\Omega_j})$.
\par The triangulation $\TT_h^{\Omega_j}$ induces a conforming triangulation of the boundary 
which is denoted by $\TT_h^{\Omega_j}|_{\Gamma_j}$.
Additionally, we define the discrete space $\PP^0(\TT_h^{\Omega_j}|_{\Gamma_j}) = \set{ \psi}{ \psi|_E \,\text{constant for all }
E\in\TT_h^{\Omega_j}|_{\Gamma_j} }$ of all piecewise constant functions on the boundary.
\par Finally, for Banach spaces $X$ and $Y$, $L(X,Y)$ denotes the space of all linear and continuous operators $S:X\to Y$.
\subsubsection{Integral operators and mapping properties}
\label{section:bio}
The following applications need two integral operators for either $\Gamma_j$, namely the double-layer potential $\widetilde K_j$ and the simple-layer potential
$\widetilde V_j$, which formally read
\begin{align*}
  (\widetilde K_j v)(x) &= \frac1{4\pi} \int_{\Gamma_j} \frac{(x-y)\cdot\normal(y)}{|x-y|^3}
  v(y)\,d\Gamma(y), \\
  (\widetilde V_j \phi)(x) &= \frac1{4\pi} \int_{\Gamma_j} \frac1{|x-y|} \phi(y) \,d\Gamma(y),
\end{align*}
for all $x\in\R^3\backslash\Gamma_j$. These operators may be extended to bounded, linear operators
$\widetilde K_j : H^{1/2}(\Gamma_j) \to H^1(\R^3\backslash\Gamma_j)$ and $\widetilde V_j :
H^{-1/2}(\Gamma_j) \to \revision{H^1_{\ell oc}(\R^3)}$, see, e.g., Refs.~\refcite{hw}, \refcite{mclean}, \refcite{sauschwa,s}. 
There holds
\begin{align}\label{eq:intop_laplace}
  \Delta \widetilde K_j v = 0 = \Delta \widetilde V_j \phi \quad\text{on }\R^3\backslash\Gamma_j
  \quad\text{and}\quad \widetilde K_jv, \widetilde V_j\phi \in C^\infty(\R^3\backslash\Gamma_j).
\end{align}
Via restriction to the boundary $\Gamma_j$, one obtains
\begin{align*}
 \gamma_j^{\rm int}\widetilde K_j v = (K_j-1/2)v 
 \quad\text{and}\quad 
 \gamma_j^{\rm int}\widetilde V_j\phi = V_j\phi,
\end{align*}
where the operators $K_j:H^{1/2}(\Gamma_j) \to H^{1/2}(\Gamma_j)$ and $V_j : H^{-1/2}(\Gamma_j) \to
H^{1/2}(\Gamma_j)$ coincide formally with $\widetilde K_j$ and $\widetilde V_j$, but are evaluated
on the boundary $\Gamma_j$. There hold the following jump properties across $\Gamma_j$, see, e.g., Ref.~\refcite{sauschwa}, Theorem~3.3.1:
\revision{\begin{align*}
 \gamma_j^{\rm ext}\widetilde K_j v - \gamma_j^{\rm int}\widetilde K_j v &= v, 
 & \delta_j^{\rm ext}\widetilde K_j v - \delta_j^{\rm int}\widetilde K_j v &= 0, 
 \\
 \gamma_j^{\rm ext}\widetilde V_j \phi - \gamma_j^{\rm int}\widetilde V_j\phi &= 0, 
 & \delta_j^{\rm ext} \widetilde V_j\phi - \delta_j^{\rm int} \widetilde V_j \phi &= -\phi.
\end{align*}}
\subsection{\revision{Strongly} monotone operators}
\label{sec:monotone}
We consider the frame of the Browder-Minty 
theorem, see Ref.~\refcite{zeidler}, Section~26.2: Let $X$ be a separable Hilbert space \revision{with dual space $X^*$}, 
$A: X \rightarrow X^*$ be a \revision{strongly} monotone and hemicontinuous (non-linear) operator, and $b \in X^*$.
Under these assumptions, the Browder-Minty theorem states that the operator equation
\begin{align}\label{eq:browderminty}
A w = b
\end{align}
has a unique solution $w \in X$. Arguing as in the original proof, one has the 
following: For $h>0$, let $X_h \subseteq X$ be finite dimensional subspaces of 
$X$ with $X_h \subseteq X_{h'}$ for $h>h'$ and
$\overline{\bigcup_{h>0}X_h} = X$. Let $b_h \in X_h^*$. 
Then, the Galerkin formulation
\begin{align*}
\dual{A w_h}{v_h}_{X^* \times X} = \dual{b_h}{v_h}_{X^* \times X} \quad \text{ for all } v_h \in X_h
\end{align*}
admits a unique solution $w_h \in X_h$. Provided 
$\norm{b_h}{X_h^*} \le M < \infty$ for all $h>0$, the sequence of 
Galerkin  solutions is bounded, i.e., $\norm{w_h}{X_h}\le C < \infty$ for all
$h>0$, and the $h$-independent constant $C>0$ depends only on $M$ and the 
coercivity \revision{constant} of $A$. In particular, the sequence $\left\{w_h\right\}_{h>0}$ \revision{admits a 
weakly convergent subsequence in $X$ with limit} $w \in X$.
If $b_h \to b$ strongly in $X^*$ for $h \to 0$, this limit solves the operator equation~\eqref{eq:browderminty}. 
Finally, \revision{strong} monotonicity implies that there even holds strong convergence $w_h \to w$ in $X$ of the entire sequence.
\par This framework is now used in the following lemma which guarantees the assumptions~\eqref{assumption:chi:bounded}--\eqref{assumption:chi:convergence} of Theorem~\ref{theorem} for certain energy contributions:
\begin{lemma}\label{prop:nonlin:conv}
\revision{Suppose that $X$ and $A:X\to X^*$ satisfy the foregoing assumptions.}
Let $Y$ be a Banach space and let $S, S_h \in L\left(X,\L^2(\revision{\Omega_1})\right)$, 
and $R, R_h \in L\big(\revision{\H^{1-\eps}(\Omega_1)} \times Y, X^*\big)$ 
\revision{for some $0\le\eps\le1$}
with 
\begin{align}
S_h x \weakto Sx
&\quad\text{weakly in $\L^2(\Omega_1)$ for all $x \in X$},
\label{prop:nonlinear:S}\\
R_h (\n, y) \rightarrow R (\n, y)
&\quad\text{strongly in $X^*$ for all $\n \in \revision{\H^{1-\eps}(\Omega_1)}, y \in Y$},
\label{prop:nonlinear:R}
\end{align}
and $\operator:=SA^{-1}R:\H^1(\Omega_1)\times Y\to\L^2(\Omega_1)$.
For $h > 0$, $\n\in\revision{\H^1(\Omega_1)}$, and $y \in Y$, define $\operator_h(\n, y):= S_h u_h$, where $u_h$ is the unique solution of 
\begin{align}\label{eq:multiscale:galerkin}
\dual{A u_h}{v_h}_{X^*\times X} = \dual{R_h (\n,y)}{v_h}_{X^*\times X} \quad \text{ for all } v_h \in X_h.
\end{align}
\revision{For all $y\in Y$, it then} holds that 
\begin{align}\label{eq:pi_bounded}
 \norm{\operator_h(\n, y)}{\L^2(\Omega_1)} 
 \le \c{multiscale}\revision{\,(1+\norm{\nabla\n}{\L^2(\Omega)})}.
 \end{align}
for all $\n \in \revision{\H^1(\Omega_1)}$ with $|\n| \le 1$ and for all $h>0$. The constant $\setc{multiscale} > 0$ does not depend \revision{on $h$} and $\n$, but only on \revision{$A$, $\norm{y}Y$},
$\Omega_1$, \revision{and the operators $S_h$ and $R_h$.} 
Moreover, suppose that $\revision{\norm{\m_{hk}^-}{L^2(\H^1)}+\norm{\zeta_{hk}^-}{L^\infty(Y)}}\le\setc{dp:bounded}$ and
$(\m_{hk}^-, \zeta_{hk}^-) \rightarrow (\m, \zeta)$ \revision{strongly} in $L^2\big([0,T]; \L^2(\Omega_1) \times Y\big) = L^2(\L^2(\Omega_1) \times Y)$ \revision{for some subsequence as $(h,k) \rightarrow (0,0)$. Then, 
\begin{align}\label{eq:lemma}
\operator_h(\m_{hk}^-, \zeta_{hk}^-) \rightharpoonup \operator(\m, \zeta)
\quad\text{weakly in $\L^2(\Omega_T)$}
\end{align}
 for the same subsequence.}
\end{lemma}
\begin{proof}
The Banach-Steinhaus theorem implies uniform boundedness \revision{of the operator norms
$C_S := \sup_{h>0}\norm{S_h:X\to\L^2(\Omega_1)}{}<\infty$ and $C_R:= \sup_{h>0}\norm{R_h:\H^{1-\eps}(\Omega_1)\times Y\to X^*}{}<\infty$.} 
For fixed $\n\in\revision{\H^1(\Omega_1)}$ 
with $|\n|\le 1$, 
$y \in Y$, and $b_h:=R_h(\n,y)$, this implies
\begin{align*}
\norm{b_h}{\revision{X^*}} \le C_R \norm{(\n, y)}{\revision{\H^{1-\eps}(\Omega_1)}\times Y} \lesssim \big(\revision{\norm{\n}{\H^1(\Omega_1)}} + \revision{\norm{y}Y}\big)=: M < \infty.
\end{align*} 
\revision{Strong monotonicity of $A$ shows
\begin{align*}
 \norm{u_h}{X}^2 \lesssim \dual{Au_h-A(0)}{u_h}_{X^*\times X}
 &= \dual{b_h-A(0)}{u_h}_{X^*\times X}
 \\&\lesssim \norm{b_h-A(0)}{X^*}\norm{u_h}X.
\end{align*}
Thus, we infer with $|\n|\le1$
\begin{align*}
\norm{u_h}{X}\lesssim 
\norm{\nabla\n}{\L^2(\Omega_1)} + |\Omega_1|^{1/2} + \norm{y}Y
+ \norm{A(0)}{X^*}
\lesssim1+\norm{\nabla\n}{\L^2(\Omega_1)},
\end{align*}
where the hidden constant $C>0$ depends only on $A$, $C_R$, and $\norm{y}Y$}.
Consequently, this proves~\eqref{eq:pi_bounded} with $\c{multiscale}=CC_S$.

Next, we show that $\operator_h(\n_h, y_h) \weakto \operator_h(\n, y)$ weakly 
in $\L^2(\Omega_1)$ as $h\to0$ provided that $(\n_h, y_h) \to (\n, y)$ strongly in 
$\revision{\H^{1-\eps}(\Omega_1)} \times Y$.
Assumption~\eqref{prop:nonlinear:R} and the uniform boundedness of $R_h$ imply that
$R_h(\n_h,y_h) = R_h(\n,y) - R_h\big((\n-\n_h, y-y_h)\big) \to R(\n,y)$ strongly in $X^*$ as $h\to0$.
Therefore, the Browder-Minty theorem for \revision{strongly} monotone operators guarantees 
$u_h\to u$ strongly in $X$, where $u = A^{-1}R(\n,y)$
and $u_h\in X_h$ solves~\eqref{eq:multiscale:galerkin} with $(\n, y)$ replaced by $(\n_h, y_h)$.
The convergence assumption~\eqref{prop:nonlinear:S} and the uniform boundedness of $S_h$ thus show
$\operator_h(\n_h, y_h) = S_hu_h = S_hu - S_h(u-u_h) \weakto Su = \operator(\n, y)$
weakly in $\L^2(\Omega_1)$ as $h\to0$.
\par Finally, we prove $\operator_h(\m_{hk}^-, \zeta_{hk}^-)\weakto\operator(\m, \zeta)$
\revision{weakly} in $\L^2(\Omega_T)$ \revision{for a subsequence} as $(h,k)\to(0,0)$.
To that end, we choose sequences $h_\ell\to0$, $k_\ell\to0$ such that 
$(\m_\ell,\zeta_\ell):=(\m_{h_\ell k_\ell}^-, \zeta_{h_\ell k_\ell}^-)$ converges strongly in $L^2\big(\L^2(\Omega_1) \times Y\big)$ 
to $(\m, \zeta)$. 
\revision{According to interpolation theory (see, e.g., Ref.~\refcite{bl}, Section~5), interpolation
of $\L^2(\Omega_T)=L^2(\L^2)$ and $L^2(\H^1)$ yields $L^2(\H^s)$ for all $0<s<1$. 
From strong convergence
$\m_{hk}^-\to\m$ in $L^2(\L^2)$ and boundedness $\norm{\m_{hk}^-}{L^2(\H^1)}\lesssim1$, 
we thus infer strong convergence $\m_{hk}^-\to\m$ in $L^2(\H^{1-\eps})$.}
By extracting a further subsequence (not relabeled), we may assume that
$\m_\ell(t) \to \m(t)$ strongly in $\revision{\H^{1-\eps}}(\Omega_1)$ as well as 
$\zeta_\ell(t) \to \zeta(t)$ \revision{strongly} in $Y$, for almost all times $t$.
Define $\operator_\ell:=\operator_{h_\ell}$ and let $\pphi\in\L^2(\Omega_T)$. Then,
\begin{align*}
 \dual{\operator_\ell (\m_\ell, \zeta_\ell)-\operator(\m, \zeta)}{\pphi}_{\Omega_T}
  = \int_0^T \dual{\operator_\ell(\m_\ell(t),\zeta_\ell(t))-\operator(\m(t),\zeta(t))}{\pphi(t)}_{\Omega_1}\,dt.
\end{align*}
\revision{Due to} $\operator_\ell(\m_\ell(t), \zeta_\ell(t)) \weakto \operator(\m(t), \zeta(t))$ \revision{weakly in $\L^2(\Omega_1)$} as 
$\ell\to\infty$ for almost all $t \in [0,T]$, we see
pointwise convergence of the integrand to zero. According to~\eqref{eq:pi_bounded} and the assumption \revision{$\norm{\m_\ell}{L^2(\H^1)}+\norm{\zeta_\ell}{\L^\infty(Y)}\lesssim1$},
the Lebesgue dominated convergence theorem thus proves
\begin{align*}
\dual{\operator_{\ell}(\m_{\ell},\zeta_{\ell})-\operator(\m,\zeta)}{\pphi}_{\Omega_T} \to0
 \quad\text{as }\ell\to\infty.
\end{align*}
This concludes the proof.
\end{proof}
\begin{remark}
{\rm(i)} Similar arguments as in the proof of Lemma~\ref{prop:nonlin:conv} reveal 
that strong convergence $S_hx\to Sx$ in~\eqref{prop:nonlinear:S} also results in
strong convergence $\operator_h(\m_{hk}^-,\zeta_{hk}^-)\to\operator(\m,\zeta)$ in $\L^2(\Omega_T)$ 
as $h,k\to0$.\\
{\rm(ii)} The abstract framework applies, in particular, to linear contributions 
$\operator_h = R_h$ of the effective field $\heff$, where $X=\L^2(\Omega_1)$, $Y = \{0\}$, 
and the operators $A=A_h$ as well as $S=S_h$ are just the identities. In this case, $\zeta_{hk}^- = 0$ for all $(h,k) > 0$. 
In particular, we may therefore write $\operator_h(\m_{hk}^-, \zeta_{hk}^-) = \operator_h(\m_{hk}^-)$.\\
{\rm(iii)} For the multiscale approach, we use \revision{$Y = \L^2(\Omega_2)$}, $\zeta_{hk}^- = \hext_{hk}^-$, and $\zeta = \hext$, respectively.
\end{remark}
\revision{%
\begin{remark}
Provided that $R, R_h \in L\big(\L^2(\Omega_1)\times Y, X^*\big)$ with
$R_h (\n, y) \rightarrow R (\n, y)$
strongly in $X^*$ for all $(\n,y) \in \revision{\L^2(\Omega_1)}\times Y$
in~\eqref{prop:nonlinear:R}, the assumptions on the nonlinear operator $A$
can be weakened: Instead of strong monotonicity, uniform monotonicity 
of $A$ is sufficient. Then, $\norm{b_h}{X^*}\le C_R\norm{\n}{\L^2(\Omega)}
\le C_R|\Omega|^{1/2}=:M$ proves $\norm{u_h}{X}\le C$ for some constant
$C = C(M)>0$, see
Ref.~\refcite{zeidler}, Section~26.2. The remaining part of the proof of
Lemma~\ref{prop:nonlin:conv} remains unchanged with the formal choice $\eps=1$.
\end{remark}%
}
\subsection{Application: Hybrid FEM-BEM \revision{stray field computations}}\label{section:fredkinkoehler}%
\revision{In the following, we present the hybrid FEM-BEM approaches of \textsc{Fredkin} and \textsc{Koehler}, see Ref.~\refcite{fredkinkoehler}, 
and \textsc{Garc\'ia-Cervera} and \textsc{Roma}, see Ref.~\refcite{gcr},
for the approximate computation of the stray field.
We show that it satisfies the assumptions of Lemma~\ref{prop:nonlin:conv}.}
Given any $\m\in \L^2(\Omega_1)$, the non-dimensional form of~\eqref{eq:u1} reads
\begin{align*}
 \begin{array}{rcll}
 \Delta u_1 &=& \nabla \cdot \m \quad&\text{in } \Omega_1, \\
 \Delta u_{1} &=& 0\quad&\text{in }\R^3\backslash\overline\Omega_1,\\{}
 \revision{\gamma_1^{\rm ext}u_{1}-\revision{\gamma_1^{\rm int}}u_{1}} &=& 0&\text{on }\Gamma_1,\\{}
 \revision{\delta_1^{\rm ext}u_{1}-\delta_1^{\rm int}u_{1}}
 &=& -\m\cdot\revision{\normal_1}\quad&\text{on }\Gamma_1,\\
 u_{1}(x) &=& \OO(1/|x|)&\text{as }|x|\to\infty,
\end{array}
\end{align*}
\revision{where the target for our LLG integrator is the stray field $\operator(\m)=\nabla u_1$ on $\Omega_1$.}
\subsubsection{Fredkin-Koehler approach}\label{section:stray field:cont}
\noindent
\revision{The approach of \textsc{Fredkin} and \textsc{Koehler} (Ref.~\refcite{fredkinkoehler})
relies on the superposition principle
\begin{align}\label{dp:superposition}
 u_1 = \begin{cases} 
 u_{11}+u_{12}&\text{in }\Omega_1,\\
 u_{12}&\text{in }\R^3\backslash\overline\Omega_1,
 \end{cases}
\end{align}
where $u_{11}\in H^1_*(\Omega_1)$ satisfies
\begin{align}\label{eq:fk1}
 \dual{\nabla u_{11}}{\nabla v}_{\Omega_1}
 = \dual{\m}{\nabla v}_{\Omega_1}
 \quad\text{for all }v\in H^1_*(\Omega_1)
\end{align}
and $u_{12} = \widetilde K_1 \revision{\gamma_1^{\rm int}}u_{11}\in H^1(\R^3\backslash\Gamma_1)$. Since the  integration of LLG only requires 
$u_1$ on $\Omega_1$, we note that $u_{12}\in H^1(\Omega_1)$ solves
\begin{align}\label{eq:fk2}
 \revision{\gamma_1^{\rm int}}u_{12} = (K_1-1/2)\revision{\gamma_1^{\rm int}}u_{11}
 \text{ and }
 \dual{\nabla u_{12}}{\nabla v}_{\Omega_1} = 0
 \text{ for all }v\in H^1_0(\Omega_1).
\end{align}}%
To discretize the equations~\eqref{eq:fk1}--\eqref{eq:fk2}, let 
$u_{11h}\in\SS_*^1(\TT_h^{\Omega_1})$ be the unique FE solution of
\begin{align}\label{eq:fk1h}
 \dual{\nabla u_{11h}}{\nabla v_h}_{\Omega_1}
 = \dual{\m}{\nabla v_h}_{\Omega_1}
 \text{ for all }v_h\in\SS_*^1(\TT_h^{\Omega_1}).
\end{align}
Since an FE approximation $u_{12h}\in\SS^1(\TT_h^{\Omega_1})$ of~\eqref{eq:fk2} cannot satisfy
continuous Dirichlet data $(K_1-1/2)u_{11h}$, we need to discretize \revision{them}. To
that end, let $I_h^{\Omega_1}: H^1(\Omega_1) \to \SS^1(\TT_h^{\Omega_1})$ be the Scott-Zhang
projection from Ref.~\refcite{scottzhang}. Since $I_h^{\Omega_1}$ is $H^1$-stable and preserves discrete boundary data, it induces a stable projection $I_h^{\Gamma_1}: H^{1/2}(\Gamma_1) \to \SS^1(\TT_h^{\Omega_1}|_{\Gamma_1})$ with $\revision{\gamma_1^{\rm int}}I_h^{\Omega_1}v = I_h^{\Gamma_1}(\revision{\gamma_1^{\rm int}}v)$ for all $v \in H^1(\Omega_1)$, \revision{see, e.g., Ref.~\refcite{hypsing3d}.}
\revision{Let $u_{12h}\in\SS^1(\TT_h^{\Omega_1})$ 
be the unique solution of the inhomogeneous Dirichlet problem
\begin{align}\label{eq:fk2h}
 \revision{\gamma_1^{\rm int}}u_{12h} = I_h^{\Gamma_1}(K_1-1/2)\revision{\gamma_1^{\rm int}}u_{11h}
 \text{ and }
 \dual{\nabla u_{12h}}{\nabla v_h}_{\Omega_1} = 0
 \quad\text{for all }v_h\in\SS_0^1(\TT_h^{\Omega_1}).
\end{align}
The resulting approximate stray field $\operator_h(\m)=\nabla u_{11h}+\nabla u_{12h}$ is indeed covered by our approach from Section~\ref{sec:monotone}.}

\begin{proposition}\label{prop:stray field}
The operator $\operator_h(\m) = R_h(\m) := \nabla u_{11h} + \nabla u_{12h}$ defined
via~\eqref{eq:fk1h}--\eqref{eq:fk2h} satisfies $\operator_h\in
L(\L^2(\Omega_1);\L^2(\Omega_1))$, and convergence~\eqref{prop:nonlinear:R}
towards $\operator(\m) = R(\m) :=\nabla u_1$ holds even strongly in $\L^2(\Omega_1)$. In particular, Lemma~\ref{prop:nonlin:conv} applies with $X:=\L^2(\Omega_1)$ and $Y:=\{0\}$ and guarantees the assumptions~\eqref{assumption:chi:bounded}--\eqref{assumption:chi:convergence} of Theorem~\ref{theorem}.
\end{proposition}

\begin{proof}
First, note that the FE solution $u_{11h}$ of~\eqref{eq:fk1h} is a Galerkin approximation of~\eqref{eq:fk1}. Therefore, stability \revision{and density
arguments prove $\norm{u_{11}-u_{11h}}{H^1(\Omega_1)}\to 0$ as $h\to 0$.}
%
\revision{Next, we consider the unique solution $\widetilde u_{12h}\in\SS^1(\TT_h^{\Omega_1})$ of the auxiliary problem
\begin{align*}
 \revision{\gamma_1^{\rm int}}\widetilde u_{12h} = I_h^{\Gamma_1}(K_1-1/2)\revision{\gamma_1^{\rm int}}u_{11}
 \text{ and }
 \dual{\nabla \widetilde u_{12h}}{\nabla v_h}_{\Omega_1} = 0
 \quad\text{for all }v_h\in\SS_0^1(\TT_h^{\Omega_1}).
\end{align*}
Note that $\gamma_1^{\rm int}\widetilde u_{12h} = I_h^{\Gamma_1}\gamma_1^{\rm int} u_{12}$. Therefore, the C\'ea lemma for inhomogeneous Dirichlet problems (see Prop.~2.3 in Ref.~\refcite{dirichlet3d}) and density arguments prove
\begin{align*}
 \norm{u_{12}-\widetilde u_{12h}}{H^1(\Omega_1)}
 \lesssim \min_{v_h\in\SS^1(\TT_h^{\Omega_1})}\norm{u_{12}-v_h}{H^1(\Omega_1)}
 \xrightarrow{h\to0}0.
\end{align*}}
\revision{Third, stability of the inhomogeneous Dirichlet problem provides
\begin{align*} 
 \norm{u_{12h}-\widetilde u_{12h}}{H^1(\Omega_1)}
 \lesssim \norm{\gamma_1^{\rm int}(u_{11}-u_{11h})}{H^{1/2}(\Gamma_1)}
 \lesssim \norm{u_{11}-u_{11h}}{H^1(\Omega_1)},
\end{align*}
and the triangle inequality reveals
\begin{align*}
 \norm{u_{12}-u_{12h}}{H^1(\Omega_1)}
 &\le \norm{u_{12}-\widetilde u_{12h}}{H^1(\Omega_1)}
 + \norm{u_{12h}-\widetilde u_{12h}}{H^1(\Omega_1)}
 \xrightarrow{h\to0}0.
\end{align*}}%
\revision{Finally,} the triangle inequality yields
\begin{align*}
\norm{\operator_h(\m)-\operator(\m)}{\L^2(\Omega_1)} \le \norm{\nabla(u_{11}-u_{11h})}{\L^2(\Omega_1)}
+ \norm{\nabla(u_{12}-u_{12h})}{\L^2(\Omega_1)}\to0
\end{align*}
for all $\m \in X = \L^2(\Omega_1)$. 
\revision{Together with Lemma~\ref{prop:nonlin:conv}, we conclude the proof.}
\end{proof}
\revision{\begin{remark}
Instead of the Scott-Zhang projection $I_h^{\Gamma_1}$, any Cl\'ement-type 
operator $I_h^{\Gamma}:L^2(\Gamma_1)\to\SS^1(\TT_h^{\Omega_1}|_{\Gamma_1})$ can be employed. 
The assertion of Proposition~\ref{prop:stray field} holds accordingly, see Ref.~\refcite{goldenits}, Section~4.3.
We note that Ref.~\refcite{fredkinkoehler}
employs nodal interpolation which is \emph{not} suitable for the numerical 
analysis as $H^1$-functions are not continuous, in general.
\end{remark}}

\subsubsection{Garc\'ia-Cervera-Roma approach}\label{section:strayfield2:cont}

\noindent
The approach of \textsc{Garc\'ia-Cervera} and \textsc{Roma}, see Ref.~\refcite{gcr},
relies also on the superposition~\eqref{dp:superposition}, where now 
$u_{11}\in H^1_0(\Omega_1)$ satisfies
\begin{align}\label{eq:gcr1}
 \dual{\nabla u_{11}}{\nabla v}_{\Omega_1}
 = \dual{\m}{\nabla v}_{\Omega_1}
 \quad\text{for all }v\in H^1_0(\Omega_1)
\end{align}
and $u_{12} = \widetilde V_1(\m\cdot\normal_1-\delta_1^{\rm int}u_{11})
\in H^1_{\ell oc}(\R^3)$. Note that $u_{12}\in H^1(\Omega_1)$ solves
\begin{align}\label{eq:gcr2}
 \gamma_1^{\rm int}u_{12} = V_1(\m\cdot\normal_1-\delta_1^{\rm int}u_{11})
 \text{ and }
  \dual{\nabla u_{12}}{\nabla v}_{\Omega_1} = 0
 \text{ for all }v\in H^1_0(\Omega_1).
\end{align}
To discretize~\eqref{eq:gcr1}--\eqref{eq:gcr2}, we employ the $L^2$-projection 
$\Pi_h:L^2(\Gamma_1)\to\PP^0(\TT_h^{\Omega_1}|_{\Gamma_1})$ 
as well as the Scott-Zhang projection $I_h^{\Gamma_1}$
and solve for $u_{11h}\in\SS^1_0(\TT_h^{\Omega_1})$ with
\begin{align}\label{eq:gcr1h}
 \dual{\nabla u_{11h}}{\nabla v_h}_{\Omega_1}
 = \dual{\m}{\nabla v_h}_{\Omega_1}
 \quad\text{for all }v_h\in \SS^1_0(\TT_h^{\Omega_1}) 
\end{align}
and for $u_{12h}\in\SS^1(\TT_h^{\Omega_1})$ with
\begin{subequations}\label{eq:gcr2h}
\begin{align}
 &\gamma_1^{\rm int}u_{12h} = I_h^{\Gamma_1}V_1(\Pi_h(\m\cdot\normal_1)-\partial u_{11h}/\partial\normal_1),
\\&
  \dual{\nabla u_{12h}}{\nabla v_h}_{\Omega_1} = 0
 \quad\text{for all }v_h\in \SS^1_0(\TT_h^{\Omega_1}).
\end{align}
\end{subequations}
The resulting approximate stray field $\operator_h(\m)=\nabla u_{11h}+\nabla u_{12h}$ is indeed covered by our approach from Section~\ref{sec:monotone}.
Unlike the Fredkin-Koehler approach, however, the numerical analysis is
slightly more involved, since the well-posedness of~\eqref{eq:gcr2}
requires at least that the normal trace $\m\cdot\normal_1$ exists 
in $H^{-1/2}(\Gamma_1)$ which prevents to consider $\m\in\L^2(\Omega_1)$ only.

\begin{proposition}\label{prop:gcr}
There exists some $\eps>0$ such that the operator 
$\operator_h(\m) = R_h(\m) := \nabla u_{11h} + \nabla u_{12h}$ defined via~\eqref{eq:gcr1h}--\eqref{eq:gcr2h} satisfies $\operator_h\in
L(\H^{1-\eps}(\Omega_1);\L^2(\Omega_1))$ as well as convergence~\eqref{prop:nonlinear:R} towards $\operator\in L(\H^{1-\eps}(\Omega_1);\L^2(\Omega_1))$,
$\operator(\m) = R(\m) := \nabla u_1 = \nabla u_{11}+\nabla u_{12}$. In particular, Lemma~\ref{prop:nonlin:conv} applies with $X:=\L^2(\Omega_1)$ and $Y:=\{0\}$ and guarantees the assumptions~\eqref{assumption:chi:bounded}--\eqref{assumption:chi:convergence} of Theorem~\ref{theorem}.
\end{proposition}

\begin{proof}
We argue essentially as in the proof of Proposition~\ref{prop:stray field}.
First, we see that 
\begin{align*}
 \norm{u_{11}-u_{11h}}{H^1(\Omega_1)}
 \lesssim \min_{v_h\in\SS^1_0(\TT_h^{\Omega_1})}\norm{u_{11}-v_h}{H^1(\Omega_1)}
 \xrightarrow{h\to0}0,
\end{align*}
for all $\m\in\L^2(\Omega_1)$. Moreover, for $\m\in \H^1(\Omega_1)$, elliptic 
regularity for the Dirichlet problem~\eqref{eq:gcr1} even predicts 
$u_{11}\in H^{3/2+\mu}(\Omega_1)$ and hence 
$\norm{u_{11}-u_{11h}}{H^1(\Omega_1)} = \OO(h^{1/2+\mu})$ for some $\mu>0$
which depends only on the shape of the polyhedral Lipschitz domain $\Omega_1$,
see, e.g., Ref.~\refcite{monk}, Theorem~3.8. By interpolation, these
observations yield the existence of some (small) $0<\eps<1/2$ such that 
\begin{align}\label{dp:regularity}
 u_{11}\in H^{3/2+\eps}(\Omega_1)\text{ with }
 \norm{u_{11}-u_{11h}}{H^1(\Omega_1)} = \OO(h^{1/2+\eps})
 \text{ for all } \m\in\H^{1-\eps}(\Omega_1).
\end{align}
From now on, we assume $\m\in\H^{1-\eps}(\Omega_1)$ and note that,
in particular, $\delta_1^{\rm int}u_{11} = \partial u_{11}/\partial\normal_1$ 
exists in $L^2(\Gamma_1)$. The trace inequality (e.g. Ref.~\refcite{fkmp}, 
Lemma~3.4) proves for any face
$E\in\TT_h^{\Omega_1}|_{\Gamma_1}$ with corresponding element $T\in\TT_h^{\Omega_1}$
(i.e., $E\subset\partial T\cap\Gamma_1$) that
\begin{align*}
 \norm{\delta_1^{\rm int}&u_{11} - \partial u_{11h}/\partial\normal_1}{L^2(\partial T\cap\Gamma_1)}^2\\
 &\lesssim h^{-1}\norm{\nabla(u_{11}-u_{11h})}{\L^2(T)}^2 + \norm{\nabla(u_{11}-u_{11h})}{\L^2(T)} \norm{D^2(u_{11}-u_{11h})}{\L^2(T)}.
\end{align*}
With $D^2u_{11h}=0$ on $T$, we sum over all elements $T\in\TT_h^{\Omega_1}$ 
and obtain
\begin{align*}
 \norm{\delta_1^{\rm int}&u_{11}- \partial u_{11h}/\partial\normal_1}{L^2(\Gamma_1)}^2\\
 &\lesssim h^{-1}\norm{\nabla(u_{11}-u_{11h})}{\L^2(\Omega_1)}^2 + \norm{\nabla(u_{11}-u_{11h})}{\L^2(\Omega_1)}\norm{D^2u_{11}}{\L^2(\Omega_1)}
 = \OO(h^{2\eps}).
\end{align*}
Together with the continuous inclusion 
$H^{-1/2}(\Gamma_1)\subseteq L^2(\Gamma_1)$, 
it follows $\norm{\delta_1^{\rm int}u_{11}- \partial u_{11h}/\partial\normal_1}{H^{-1/2}(\Gamma_1)}\to0$ as $h\to0$.
Let $\widetilde u_{12h}\in\SS^1(\TT_h^{\Omega_1})$ be the unique
solution of the auxiliary problem
\begin{align*}
 \gamma_1^{\rm int}\widetilde u_{12h} = I_h^{\Gamma_1}V_1(\m\cdot\normal_1-\delta_1^{\rm int}u_{11})
 \text{ and }
 \dual{\nabla \widetilde u_{12h}}{\nabla v_h}_{\Omega_1} = 0
 \quad\text{for all }v_h\in\SS_0^1(\TT_h^{\Omega_1}).
\end{align*}
Again, it holds $\gamma_1^{\rm int}\widetilde u_{12h} = I_h^{\Gamma_1}u_{12}$
and hence $\norm{u_{12}-\widetilde u_{12h}}{H^1(\Omega_1)}\to0$ as $h\to0$.
Stability of the inhomogeneous Dirichlet problem proves
\begin{align*} 
 \norm{\widetilde u_{12h}-u_{12h}}{H^1(\Omega_1)}
 &\lesssim\norm{I_h^{\Gamma_1}V_1\big((1-\Pi_h)\m\cdot\normal_1-(\delta_1^{\rm int}u_{11}-\partial u_{11h}/\partial\normal_1)\big)}{H^{1/2}(\Gamma_1)}
 \\&\lesssim
 \norm{(1-\Pi_h)\m\cdot\normal_1}{H^{-1/2}(\Gamma_1)}
 + \norm{\delta_1^{\rm int}u_{11}-\partial u_{11h}/\partial\normal_1}{H^{-1/2}(\Gamma_1)}.
\end{align*}
We already saw that the second term on the right-hand side vanishes
as $h\to0$.
For the first term, a duality argument (see, e.g., Ref.~\refcite{ccdpr}, Section~4) proves
\begin{align*}
 \norm{(1-\Pi_h)\m\cdot\normal_1}{H^{-1/2}(\Gamma_1)}
 \lesssim h^{1/2}\norm{\m\cdot\normal_1}{L^2(\Gamma_1)}
 \lesssim h^{1/2}\norm{\m}{\H^{1-\eps}(\Omega_1)},
\end{align*}
where we also used $0<\eps<1/2$ to admit a continuous trace operator
$\gamma_1^{\rm int}:\H^{1-\eps}(\Omega_1)\to\L^2(\Gamma_1)$.
Overall, we thus see
\begin{align}\label{eq:gcr:u12conv}
 \norm{u_{12}-u_{12h}}{H^1(\Omega_1)}
 \le \norm{u_{12}-\widetilde u_{12h}}{H^1(\Omega_1)}
 + \norm{\widetilde u_{12h}-u_{12h}}{H^1(\Omega_1)}
 \xrightarrow{h\to0}0.
\end{align}
The combintation of~\eqref{dp:regularity}--\eqref{eq:gcr:u12conv}
concludes $\norm{\pi(\m)-\pi_h(\m)}{\L^2(\Omega_1)}
\to0$ as $h\to0$, for all $\m\in\H^{1-\eps}(\Omega_1)$.
\end{proof}

\begin{figure}
\begin{center}
\newcommand\psbox[4][white]{\rput(#2){\rnode{#4}{%
  \psframebox[fillcolor=#1]{\tiny{\makebox{\tabular{c}#3\endtabular}}}}}}
\psset{framearc=0,shadow=true,xunit=0.8mm,yunit=0.8mm,fillstyle=solid, linestyle=none, shadowcolor=white}
\begin{pspicture}(8,-0)(135,50)
\psbox[gray!30]{30,25}{$\widetilde R$\\\qquad\qquad\qquad\qquad\\\,\\\,\\\,\\\,\\\,\\\,\\\,\\\,\\\,\\\,\\\,\\\,\\\,\\\,\\\,\\\,\\\,\\\,\\\,\\\,}{b1}
\psbox[gray!30]{74,25}{$\widetilde A^{-1}$\\\qquad\qquad\qquad\qquad\\\,\\\,\\\,\\\,\\\,\\\,\\\,\\\,\\\,\\\,\\\,\\\,\\\,\\\,\\\,\\\,\\\,\\\,\\\,\\\,}{b2}
\psbox[gray!30]{127,25}{\vphantom{$\widetilde S$}$S$\\\qquad\qquad\qquad\qquad\\\,\\\,\\\,\\\,\\\,\\\,\\\,\\\,\\\,\\\,\\\,\\\,\\\,\\\,\\\,\\\,\\\,\\\,\\\,\\\,}{b3}
\psset{framearc=0.2,shadow=true,xunit=0.8mm,yunit=0.8mm,fillstyle=solid,linestyle=solid, shadowcolor=black!55}
\psbox[green!30]{-8,25}{Input: $\m$}{input}
\psbox[green!30]{-8,5}{Input: $\f$}{input2}
\psbox[blue!30]{30,25}{Solve~\eqref{eq:fk1} to obtain\\{\color{red}\textbf{$\mathbf{u_{11}}$ on $\mathbf{\Omega_1}$}}}{one}
\psbox[blue!30]{30,45}{Solve~\eqref{eq:fk5} to obtain\\{\color{red}\textbf{$\mathbf{u_1}$ on $\mathbf{\Omega_2}$}}}{two}
\psbox[blue!30]{74,25}{Solve~\eqref{eq:jn1} to obtain\\%
{\color{red}\textbf{$\mathbf{u}$ on $\mathbf{\Omega_2}$}} and {\color{red}\textbf{$\mathbf{\delta_2^{\rm ext}u_2}$ on $\Gamma_2$}}}{four}
\psbox[blue!30]{127,25}{Solve~\eqref{eq:jn3} to obtain\\{\color{red}\textbf{$\mathbf{u_2}$ on $\mathbf{\Omega_1}$}}}{five}
\psbox[blue!30]{30,5}{Solve~\eqref{eq:uext:nondim} to obtain\\{\color{red}\textbf{$\mathbf{\uext}$ on $\mathbf{\Omega_2}$}}}{three}
\psbox[red!30]{127,5}{Output: $\mathbf{\operator(\m, \f)=\nabla u_2}$ on $\Omega_2$}{output}
\psset{nodesep=3pt,arrows=c->,shortput=nab,shadow=false,linewidth=0.05,arrowinset=0,arrowlength=0.8}
\ncline{input}{one}
\ncline{one}{two}
\ncline{two}{four}
\ncline{two}{five}
\ncline{three}{four}
\ncline{three}{five}
\ncline{four}{five}
\ncline{five}{output}
\ncline{input2}{four}
\ncline{input2}{three}
\end{pspicture}

\end{center}
\caption{Overview on the computation of $\operator(\m, \f) = \nabla u_2$ on $\Omega_1$.} \label{fig:u2}
\end{figure}
\subsection{Application: Multiscale approach for total magnetic field}
\label{sec:multi}
We aim to apply Lemma~\ref{prop:nonlin:conv} to the model problem posed in Section~\ref{sec:maxwell}, i.e., the
computation of $\operator(\m,\hext) = \nabla u_2$ on $\Omega_1$. 
In the following, we consider the subproblems needed for the computation of $\nabla u_2$
as well as their discretizations. An overview illustration is given in Figure~\ref{fig:u2}.
Throughout this section, we let
\begin{itemize}
\item $X:=H^{-1/2}(\Gamma_2)\times H^1(\Omega_2)$,
\item $Y:=\L^2(\Omega_2)$.
\end{itemize}
We recall that $H^{-1/2}(\Gamma_2)$ is the dual space of the trace space
$H^{1/2}(\Gamma_2)$ and that $\widetilde H^{-1}(\Omega_2)$ is the dual space of 
$H^1(\Omega_2)$, where duality is understood according to the respective $L^2$-scalar 
products. In particular, the dual space of $X$ is 
$X^* = H^{1/2}(\Gamma_2)\times \widetilde H^{-1}(\Omega_2)$.
\subsubsection{Continuous formulation}
To compute $\nabla u_2$ on $\Omega_1$, we proceed as implicitly outlined in Section~\ref{sec:maxwell}.
For a magnetization $\m \in \L^2(\Omega_1)$, we compute $u_{1}\in H^1(\Omega_1)$ 
as solution of the stray field operator on the microscopic part.
Recall from Section~\ref{section:fredkinkoehler} that in $\R^3\backslash\overline\Omega_1\supset\Omega_2$ it holds
$u_1 = u_{12} = \widetilde K_1 \gamma_1^{\rm int}u_{11}$ with $u_{11}\in H_*^1(\Omega_1)$ being the solution of~\eqref{eq:fk1}.
According to~\eqref{eq:intop_laplace},
$u_1$ on $\Omega_2$ thus solves the inhomogeneous Dirichlet problem
\begin{align}\label{eq:fk5}
 \gamma_2^{\rm int}u_1 
 = \gamma_2^{\rm int}\widetilde K_1 \gamma_1^{\rm int}u_{11}
 \text{ and }
 \dual{\nabla u_1}{\nabla v}_{\Omega_2}=0
 \text{ for all }v\in H^1_0(\Omega_2).
\end{align}
Recall $\nabla\cdot\hext=0$ from~\eqref{eq:Hext:div}, whence
$\dual{\hext\cdot\normal_2}{\gamma_2^{\rm int}v}_{\Gamma_2}
=\dual{\hext}{\nabla v}_{\Omega_2}$ for all $v\in H^1(\Omega_2)$.
 For the auxiliary potential $\uext\in H^1_*(\Omega_2)$,
the non-dimensional weak formulation of~\eqref{eq:uext} reads
\begin{align}\label{eq:uext:nondim}
  \dual{\nabla\uext}{\nabla v}_{\Omega_2}
  = -\dual{\hext}{\nabla v}_{\Omega_2}
 \text{ for all }v\in H^1_*(\Omega_2).
\end{align}
\par In the next step, we then compute the total magnetostatic potential 
$u = u_1 + u_2 + \uext$ on the macroscopic domain $\Omega_2$. With 
$\widetilde\chi(|\nabla u|) = \chi\big(M_s|\hext-\nabla u_1-\nabla u_2|\big)$,
the non-dimensional form 
of~\eqref{eq:u:omega2} 
reads
\begin{subequations}\label{eq:jn0}
\begin{align}
  \label{eq_:u:omega2:interior}
  \nabla\cdot\big( (1+\widetilde\chi(|\nabla u|))\nabla u\big) &= 0
  \hspace*{29.1mm}\text{in }\Omega_2, \\
  \Delta u_2 &= 0 
  \hspace*{29.1mm}\text{in }\R^3\backslash\overline\Omega_2, \\
  \label{eq_:u:omega2:jumpu}
  \gamma_2^{\rm ext}u_2 - \gamma_2^{\rm int}u &= -\gamma_2^{\rm int}(u_1 +\uext)
  \hspace*{5.2mm}\text{on }\Gamma_2, \\
  \label{eq_:u:omega2:jumpdu}
  \delta_2^{\rm ext}u_2 - (1+\widetilde\chi(|\nabla u|))\nabla u\cdot\normal_2 &= \hext\cdot\normal_2 - \delta_2^{\rm int}u_1
  \hspace*{5.6mm}\quad\text{on }\Gamma_2, \\
  u_2(x) &= \OO(1/|x|) 
  \hspace*{18.0mm}\text{as }|x|\to\infty.
\end{align}
\end{subequations}
Let $V_2 : H^{-1/2}(\Gamma_2) \rightarrow H^{1/2}(\Gamma_2)$ and $K_2:H^{1/2}(\Gamma_2) \rightarrow H^{1/2}(\Gamma_2)$ denote the simple-layer potential and 
the double-layer potential with respect to $\Gamma_2$ (see Section~\ref{section:bio}).
The transmission problem~\eqref{eq:jn0}
is then equivalently stated by means of the Johnson-N\'ed\'elec 
coupling from Ref.~\refcite{johnson-nedelec},
\begin{subequations}\label{eq:jn1}
\begin{align}
 \dual{(1+\widetilde\chi(|\nabla u|))\nabla u}{\nabla v}_{\Omega_2}
 -\dual{\phi}{\gamma_2^{\rm int}v}_{\Gamma_2}
 &= \dual{\delta_2^{\rm int}u_1}{\gamma_2^{\rm int}v}_{\Gamma_2}
 \!-\! \dual{\hext}{\nabla v}_{\Omega_2},
\\
 V_2\phi + (1/2-K_2)\gamma_2^{\rm int}u &= (1/2-K_2)\gamma_2^{\rm int}(u_1+\uext),
\end{align}
\end{subequations}
for all $v\in H^1(\Omega_2)$, see Ref.~\refcite{affkmp} for the non-linear case
and Refs.~\refcite{johnson-nedelec}, \refcite{sayas09} for the linear one.
The coupling formulation~\eqref{eq:jn1} provides the total potential $u\in H^1(\Omega_2)$ as well as 
the exterior normal derivative $\phi = \delta_2^{\rm ext}u_2\in H^{-1/2}(\Gamma_2)$. 
Existence and uniqueness of the solution $(\phi,u)\in X = H^{-1/2}(\Gamma)\times H^1(\Omega)$ of~\eqref{eq:jn1} hinges strongly on the material law 
$\widetilde\chi$ and will be discussed in Section~\ref{section:jn} 
below.

Since $u_2$ solves $-\Delta u_2 = 0$ in $\R^3\backslash\overline\Omega_2$, $u_2$ can be
computed by means of the representation formula
\begin{align}\label{eq:jn2}
 u_2 = -\widetilde V_2\delta_2^{\rm ext}u_2 + \widetilde K_2\gamma_2^{\rm ext}u_2
 \quad\text{in }\R^3\backslash\overline\Omega_2 \supset \Omega_1,
\end{align}
see, e.g., Ref.~\refcite{sauschwa}, Theorem~3.1.6.
To lower the computational cost for the later implementation, we will, however, not use 
the representation formula~\eqref{eq:jn2} on $\Omega_1$, but only on $\Gamma_1$ and solve an 
inhomogeneous Dirichlet problem instead. 
It holds $\gamma_2^{\rm ext}u_2 = \gamma_2^{\rm int}u_2
= \gamma_2^{\rm int}(u - u_1 - \uext)$. With 
$\phi=\delta_2^{\rm ext}u_2$ on $\Gamma_2$, we obtain
\begin{subequations}\label{eq:jn3}
\begin{align}
 -\Delta u_2 &= 0 \hspace*{58.5mm}\text{in }\Omega_1,\\
 \gamma_1^{\rm int}u_2 &= \gamma_1^{\rm int}\big(-\widetilde V_2\phi + \widetilde K_2\gamma_2^{\rm int}(u - u_1 - \uext)\big)\quad\text{on }\Gamma_1.
\end{align}
\end{subequations}
\subsubsection{Discrete formulation}
As for the stray field, we solve~\eqref{eq:fk1h} to obtain an approximation 
$u_{11h} \in\SS_*^1(\TT_h^{\Omega_1})$ of $u_{11}$. To discretize~\eqref{eq:fk5},
let $u_{1h}\in\SS^1(\TT_h^{\Omega_2})$ solve
\begin{align}\label{eq:fk5:disc}
 \gamma_2^{\rm int}u_{1h} = I_h^{\Gamma_2} K_1 \gamma_1^{\rm int}u_{11h} 
 \text{ and }
 \dual{\nabla u_{1h}}{\nabla v_h}_{\Omega_2}= 0 
 \text{ for all }v_h\in\SS_0^1(\TT_h^{\Omega_2}).
\end{align}
The discrete version of~\eqref{eq:uext:nondim} reads as follows:
Let $u_{{\rm app},h} \in \SS_*^1(\TT_h^{\Omega_2})$ solve
\begin{align}\label{eq:uext:disc}
 \dual{\nabla u_{{\rm app},h}}{\nabla v_h}_{\Omega_2}
 = - \dual{\hext}{\nabla v_h}_{\Omega_2}
 \quad\text{for all }v_h \in \SS_*^1(\TT_h^{\Omega_2}).
\end{align}
For the numerical solution of~\eqref{eq:jn1}, we compute $(\phi_h,u_h)\in X_h:=\PP^0(\TT_h^{\Omega_2}|_{\Gamma_2})\times \SS^1(\TT_h^{\Omega_2})$ 
such that
\begin{align}\label{eq:jn:disc}
\begin{split}
  \dual{(1+\widetilde\chi(|\nabla u_h|))\nabla u_h}{\nabla v_h}_{\Omega_2}  
  - \dual{\phi_h}{v_h}_{\Gamma_2}
  &= \dual{\partial u_{1h}/\partial\normal_2}{v_h}_{\Gamma_2}
  \!-\!\dual{\hext}{\nabla v_h}_{\Omega_2}, \\
  \dual{V_2\phi_h + (1/2-K_2)u_h}{\psi_h}_{\Gamma_2} 
 &= \dual{(1/2-K)(u_{1h} + u_{{\rm app},h})}{\psi_h}_{\Gamma_2}
\end{split}
\end{align}
for all $(\psi_h,v_h)\in X_h$. Existence and uniqueness of $(\phi_h,u_h)$
is discussed in Section~\ref{section:jn} below.
To discretize~~\eqref{eq:jn3}, let $u_{2h} \in \SS^1(\TT_h^{\Omega_1})$ solve
\begin{align}\label{eq:jn3:disc}
\begin{split}
&\gamma_1^{\rm int}u_{2h} = I_h^{\Gamma_1}\gamma_1^{\rm int}\big(-\widetilde V_2\phi_h + \widetilde K_2\gamma_2^{\rm int}(u_h - u_{1h} - u_{{\rm app},h})\big)
\\
&\dual{\nabla u_{2h}}{\nabla v_h}_{\Omega_1} = 0 \quad \text{ for all } v_h \in \SS^1_0(\TT_h^{\Omega_1}).
\end{split}
\end{align}

\subsubsection{Operator formulation}
\label{section:SAR}

With respect to the abstract notation of Lemma~\ref{prop:nonlin:conv},
the solutions of the problems~\eqref{eq:fk5}--\eqref{eq:uext:nondim} and~\eqref{eq:fk5:disc}--\eqref{eq:uext:disc} give rise to
the continuous linear operators
\begin{align}
\begin{split}\label{eq:multiscale:R}
 \widetilde R,\widetilde R_h&:\H^1(\Omega_1)\times \L^2(\Omega_2)
 \to H^{1/2}(\Gamma_2)\times \widetilde H^{-1}(\Omega_2),\\
 \widetilde R(\m,\hext) 
 &:=
 \big((1/2-K_2)\gamma_2^{\rm int}(u_1+\uext),(\gamma_2^{\rm int})^*\delta_2^{\rm int} u_1 -\nabla^*\hext\big),\\
  \widetilde R_h(\m,\hext) 
 &:=
 \big((1/2-K_2)\gamma_2^{\rm int}(u_{1h}+u_{{\rm app},h}),(\gamma_2^{\rm int})^*\partial u_{1h}/\partial\normal_2 -\nabla^*\hext\big),
\end{split}
\end{align}
where $(\gamma_2^{\rm int})^*:H^{-1/2}(\Gamma_2)\to \widetilde H^{-1}(\Omega_2)$ 
denotes the adjoint of the trace operator 
$\gamma_2^{\rm int}:H^1(\Omega_2)\to H^{1/2}(\Gamma_2)$
and $\nabla^*:\L^2(\Omega_2)\to\widetilde H^{-1}(\Omega_2)$ is the adjoint gradient.
Note that $\widetilde R$, $\widetilde R_h$ are also well-defined and bounded operators 
on $\L^2(\Omega_1)\times\L^2(\Omega_2)$ and hence by interpolation, for all $0<s<1,$ also
on $\H^s(\Omega_1)\times\L^2(\Omega_2)$.

The left-hand side of the coupling formulation~\eqref{eq:jn1} gives rise to the non-linear operator
\begin{align}\label{eq:multiscale:A}
\begin{split}
 &\widetilde A:H^{-1/2}(\Gamma_2)\times H^{1}(\Omega_2) 
 \to H^{1/2}(\Gamma_2)\times \widetilde H^{-1}(\Omega_2)
\end{split}
\end{align}
and is then equivalently stated by
\begin{align}\label{eq2:jn}
 \widetilde A(\phi,u) = \widetilde R(\m,\hext).
\end{align}
Note that the FEM-BEM coupling~\eqref{eq:jn:disc} takes the abstract form
\begin{align}\label{eq2:jn:disc}
 \dual{\widetilde A(\phi_h,u_h)}{(\psi_h,v_h)}_{X^*\times X}
 = \dual{\widetilde R_h(\m,\hext)}{(\psi_h,v_h)}_{X^*\times X}
\end{align}
for all 
$(\psi_h,v_h)\in X_h:=\PP^0(\TT_h^{\Omega_2}|_{\Gamma_2})\times \SS^1(\TT_h^{\Omega_2})$. In the subsequent Section~\ref{section:jn}, we comment on 
the existence and uniqueness of the solutions
of~\eqref{eq2:jn}--\eqref{eq2:jn:disc}.

Finally, the solution of~\eqref{eq:jn3} resp.\ its discretization~\eqref{eq:jn3:disc} give rise to the continuous linear operators
\begin{align}\label{eq:multiscale:S}
\begin{split}
 &S,S_h:H^{-1/2}(\Gamma_2)\times H^{1}(\Omega_2)\to\L^2(\Omega_1),
 \\
 &S(\phi,u) := \nabla u_2,
 \qquad
 S_h(\phi_h,u_h) := \nabla u_{2h}.
\end{split}
\end{align}
Overall, it holds 
\begin{align}\label{kotz:multiscale}
 \operator(\m, \hext) := S\widetilde A^{-1}\widetilde R(\m,\hext) = \nabla u_2
 \quad\text{and}\quad
 \operator_h(\m,\hext) := S_h(\phi_h,u_h) = \nabla u_{2h}
\end{align}
where 
$(\phi_h,u_h)\in X_h:=\PP^0(\TT_h^{\Omega_2}|_{\Gamma_2})\times \SS^1(\TT_h^{\Omega_2})$
solves~\eqref{eq:jn:disc} resp.~\eqref{eq2:jn:disc}.


\begin{remark}
Note that the formal definition of the operator $S$ (resp.\ $S_h$) once 
again requires the solution
of~\eqref{eq:fk5}--\eqref{eq:uext:nondim} 
(resp.~\eqref{eq:fk5:disc}--\eqref{eq:uext:disc}) to provide 
$\gamma_2^{\rm int}(u_1+\uext)$ on the right-hand side of~\eqref{eq:jn3}
(resp.~\eqref{eq:jn3:disc} with according discrete traces).
Theoretically, this can be dealt with by considering the extended operators
\begin{align*}
&\widehat R\big(\m,\hext\big)
= \big(\widetilde R(\m,\hext),\gamma_2^{\rm int}(u_1+\uext)\big),\\
&\widehat A\big(\phi,u,\gamma_2^{\rm int}(u_1+\uext)\big)
=  \big(\widetilde A(\phi,u),\gamma_2^{\rm int}(u_1+\uext)\big)\\
&\widehat S\big(\phi,u,\gamma_2^{\rm int}(u_1+\uext)\big)
= \nabla u_2.
\end{align*}%
Then, $\widehat S$ and $\widehat R$ are still linear and continuous. Provided
$A$ satisfies the assumptions of the Browder-Minty theorem for strongly
monotone operators, the inverse of $\widehat A$ is well-defined and
continuous so that (an obvious extension of) Lemma~\ref{prop:nonlin:conv}
still applies.
\end{remark}

\subsubsection{Well-posedness of Johnson-N\'ed\'elec coupling}
\label{section:jn}
The following lemma provides sufficient conditions such that the non-linear part of~\eqref{eq:jn1} is
strongly monotone and Lipschitz continuous~\eqref{eq:kotz}. The elementary proof is left 
to the reader.
\begin{lemma}\label{lemma:lipAmonA}
  Let $\widetilde\chi: \R_{\geq0} \rightarrow \R$ be a continuous function
  such that the function
  \begin{align*}
    g:\R_{\geq0} &\rightarrow\R, \quad g(t) =  t+\widetilde\chi(t)t
  \end{align*}
  is differentiable and fulfils
  \begin{align}\label{eq:gproperty}
    g'(t) \in [\gamma,L] \quad\text{for all }t\geq0
  \end{align}
  with constants $L\geq\gamma>0$.
  Then, the (non-linear) operator
  \begin{align*}
   \A : \L^2(\Omega_2) &\rightarrow \L^2(\Omega_2), \quad \A\ww = (1+\widetilde\chi(|\ww|))\ww 
  \end{align*}
  is Lipschitz continuous and strongly monotone, i.e., there holds
  \begin{align}\label{eq:kotz}
    L^{-2}\,\norm{\A\uu-\A\vv}{\L^2(\Omega_2)}^2 &\leq \norm{\uu-\vv}{\L^2(\Omega_2)}^2
    \le\gamma^{-1}\,\dual{\A\uu-\A\vv}{\uu-\vv}_{\Omega_2}
  \end{align}%
  for all $\uu,\vv\in\L^2(\Omega_2)$.
\qed
\end{lemma}

We stress that the operator $\widetilde A$ from~\eqref{eq:jn1} resp.~\eqref{eq:multiscale:A} is \emph{not} strongly monotone as, e.g., the left-hand side of~\eqref{eq:jn1} is zero for $(\phi,u)=(0,1)$. 
To overcome this problem, we define the linear operator 
\begin{align}\label{eq:multiscale:L}
  L: X^*\to X^*,\quad Lx^* := x^* + \dual{x^*}{(1,0)}_{X^*\times X} \dual{\widetilde A(\cdot,\cdot)}{(1,0)}_{X^*\times X},
\end{align}
where $1\in\PP^0(\TT_h^{\Omega_2}|_{\Gamma_2})$ denotes the constant function.
As observed in Ref.~\refcite{affkmp}, Section~4, the Johnson-N\'ed\'elec coupling equations can then be equivalently rewritten as follows:

\begin{lemma}\label{prop:equivalenceA}
The operator $L: X^* \to X^*$ from~\eqref{eq:multiscale:L} is well-defined, linear, and continuous. Let $\widetilde A$ be the operator from~\eqref{eq:jn1} resp.~\eqref{eq:multiscale:A}. Define $A := L\widetilde{A}$.
Let $X_\star$ be a closed subspace of $X=H^{-1/2}(\Gamma_2)\times H^1(\Omega_2)$
with $(1,0)\in X_\star$. Then, for any $\tilde x^*\in X^*$ and $x^*:=L\tilde x^*$,
the pair $(\phi_\star,u_\star) \in X_\star$ solves the operator formulation
\begin{align*}
 \dual{\widetilde A(\phi_\star, u_\star)}{(\psi_\star,v_\star)}_{X^*\times X}
 = \dual{\tilde x^*}{(\psi_\star,v_\star)}_{X^*\times X}
 \quad\text{for all }(\psi_\star,v_\star)\in X_\star
\end{align*}
if and only if
  \begin{align*}
    \dual{A (\phi,u)}{(\psi_\star,v_\star)}_{X^*\times X} = \dual{x^*}{(\psi_\star,v_\star)}_{X^*\times X}
 \quad\text{for all }(\psi_\star,v_\star)\in X_\star.
  \end{align*}
Under the assumptions of Lemma~\ref{lemma:lipAmonA} with $\gamma>1/4$, the operator $A = L\widetilde A$ 
  is Lipschitz continuous and strongly monotone.
  In particular, it fulfils the assumptions of the Browder-Minty theorem for 
  strongly monotone operators.
  In this case, $A$ as well as $\widetilde A$ are, in particular, invertible, and 
  $\widetilde A^{-1}\tilde x^* = A^{-1}x^*$.
    \qed
\end{lemma}

For $\gamma>1/4$, the preceding lemma applies to $X_\star = X = H^{-1/2}(\Gamma_2)\times H^1(\Omega_2)$ as well as $X_\star = X_h = \PP^0(\TT_h^{\Omega_2}|_{\Gamma_2})
\times\SS^1(\TT_h^{\Omega_2}|_{\Gamma_2})$ and thus proves 
that~\eqref{eq2:jn} as well as~\eqref{eq2:jn:disc} admit unique solutions.
\par Finally, we give some examples of material laws $\widetilde\chi$, covered by Lemma~\ref{prop:equivalenceA}.
\begin{remark}\label{rem:materiallaw}
  (i) Consider the material law 
  \begin{align*}
    \widetilde\chi(t) = \c{tanh1}\tanh(\c{tanh2} t)/t\quad\text{for }t>0,
    \quad\widetilde\chi(0) = \c{tanh1}\c{tanh2}
  \end{align*}
  with dimensionless constants
  $\setc{tanh1},\setc{tanh2}>0$. Then, $g(t) = t + \c{tanh1}\tanh{\c{tanh2}t}$ fulfils~\eqref{eq:gproperty}
  with $\gamma = 1$ and $L = 1+\c{tanh1}\c{tanh2}$.
  \\ (ii)
  According to Ref.~\refcite{rzmp81}, it is reasonable to approximate the magnetic susceptibility in terms of a
  rational function, e.g.,
  \begin{align*}
    \widetilde\chi(t) = \frac{\c{chit1} + \c{chit2}t }{1+\c{chit3}t + \c{chit4}t^2 }
  \end{align*}
  with certain, material-dependent constants
  $\setc{chit1},\setc{chit2},\setc{chit3},\setc{chit4}>0$. For typical materials, it holds~\eqref{eq:gproperty} with $\gamma=1$ and some $L>1$ that depends on $\c{chit1},\c{chit2},\c{chit3},\c{chit4}$, see Ref.~\refcite{rzmp81}, Table~1.
\end{remark}
\subsubsection{Convergence Analysis}
The main result of this section is the following proposition.
\begin{proposition}
In addition to $\hext\in L^2(\Omega_T)$, suppose that $\hext\in L^\infty(\L^2(\Omega_2))$.
Adopt the notation of Section~\ref{section:SAR} for the operators 
$\widetilde R,R_h$ from~\eqref{eq:multiscale:R}, $\widetilde A$ from~\eqref{eq:multiscale:A}
and $\widetilde S,S_h$ from~\eqref{eq:multiscale:S}. Under the assumptions of Lemma~\ref{lemma:lipAmonA} with $\gamma>1/4$, the operator 
$\operator:=S\widetilde A^{-1}\widetilde R$
and its discretization $\operator_h$ from~\eqref{kotz:multiscale} satisfy the 
assumptions~\eqref{assumption:chi:bounded}--\eqref{assumption:chi:convergence}
of Theorem~\ref{theorem}.
\end{proposition}
\begin{proof}
With Lemma~\ref{prop:equivalenceA}, there exists a linear and continuous operator 
$L:X^*\to X^*$ such that $A:=L\widetilde A$ is Lipschitz continuous and strongly
monotone. It holds $\operator = SA^{-1}R$ with $R:=L\widetilde R$ and $\operator_h(\m,\hext) 
= S_h(\phi_h,u_h)$, where $(\phi_h,u_h)$ solves with $R_h:=L\widetilde R_h$
the variational formulation
\begin{align*}
 \dual{A(\phi_h,u_h)}{(\psi_h,v_h)}_{X^*\times X}
 = \dual{R_h(\m,\hext)}{(\psi_h,v_h)}_{X^*\times X}
 \text{ for all }(\psi_h,v_h)\in X_h.
\end{align*}
Therefore, the claim follows from Lemma~\ref{prop:nonlin:conv} if we prove that
there exists some $\eps>0$ such that
\begin{itemize}
\item[(i)] $\widetilde R_h(\m,\hext) \to \widetilde R(\m,\hext)$ strongly in $X^*$ 
for all $(\m,\hext)\in\H^{1-\eps}(\Omega_1)\times\L^2(\Omega_2)$;
\item[(ii)] $\widetilde S_h x \to \widetilde Sx$ strongly in $\L^2(\Omega_1)$ for all 
$x\in X$.
\end{itemize}
To verify (i), we argue as in the proofs of Proposition~\ref{prop:stray field}
and Proposition~\ref{prop:gcr}.  First, elliptic regularity for the Neumann problem
\eqref{eq:fk1} (see, e.g., Ref.~\refcite{monk}, Theorem 3.8) provides 
some $\eps>0$ such that, for $\m\in\H^{1-\eps}(\Omega_1)$, it holds
$\norm{u_{11}-u_{11h}}{H^1(\Omega_1)} = \OO(h^{1/2+\eps})$.
Second, recall that $u_1=\widetilde K_1\gamma_1^{\rm int}u_{11}\in C^\infty(\overline\Omega_2)\subset H^2(\Omega_2)$. 
Hence, the inhomogeneous Dirichlet problem~\eqref{eq:fk1} leads to
\begin{align*}
 \norm{u_1-u_{1h}}{H^1(\Omega_2)}
 \lesssim \min_{v_h\in\SS^1(\TT_h^{\Omega_2})}\norm{u_{1}-v_h}{H^1(\Omega_2)}
 + \norm{u_{11}-u_{11h}}{H^1(\Omega_1)} = \OO(h^{1/2+\eps}).
\end{align*}
Third, arguing as in the proof of Proposition~\ref{prop:gcr}, we derive
\begin{align*}
 \norm{\delta_2^{\rm int}u_1 - \partial u_{1h}/\partial\normal_2}{H^{-1/2}(\Gamma_2)}
 =\OO(h^\eps).
\end{align*}
Fourth, the discretization of the auxiliary potential guarantees
\begin{align*}
 \norm{u_{\rm app}-u_{{\rm app},h}}{H^1(\Omega_1)}
 \lesssim \min_{v_h\in\SS^1(\TT_h^{\Omega_1})}\norm{u_{\rm app}-v_h}{H^1(\Omega_1)}
 \xrightarrow{h\to0}0.
\end{align*}
By definition~\eqref{eq:multiscale:R} of the operators $\widetilde R$ and $\widetilde R_h$, the combination of the 
foregoing three convergences proves (i).

The verification of (ii) follows along the same lines. This concludes the proof.
\end{proof}
\appendix
\section{Improved energy estimate}
\noindent Under some additional assumptions on the general field contribution $\operator$ and on the applied field $\hext$, as well as on their respective discretizations, we can derive the following physically meaningful energy estimate.
In this section, we neglect any possible dependence of $\operator$ and $\operator_h$ on a second quantity $\zeta$.
\begin{propappendix}\label{lem:energy:improved}
Let $\operator:\L^2(\Omega_1)\to\L^2(\Omega_1)$ be a linear, bounded, and self-adjoint operator, satisfying
\begin{equation} \label{eq:boundedness_l4}
\norm{\operator(\w)}{\L^4(\Omega_1)} \leq \c{l4_boundedness} \norm{\w}{\L^4(\Omega_1)} \quad \text{ for all } \w \in \L^4(\Omega_1)
\end{equation}
with a constant $\setc{l4_boundedness}>0$.
Let $\operator_h$ satisfy
\begin{equation}\label{assumption:chi:convergence2}
 \operator_h(\m_{hk}^-) \to \operator(\m)
 \quad\text{strongly in }\L^2(\Omega_T) \text{ for some subsequence}.
\end{equation}
Let the applied field $\hext \in \L^4(\Omega_1)$ be constant in time.
Assume that $\hext_h^{j+1}=\hext_h^j=\hext_h$ for all $j$, and
\begin{align}\label{assumption:f2}
 \hext_h \to \hext \quad\text{ strongly in }\L^2(\Omega_1).
\end{align}
Then, the energy
\begin{align}\label{eq:energy}
\EE\left(\m(t)\right) :=
\frac{\Cexchange}{2}\norm{\nabla \m(t)}{\L^2(\Omega_1)}^2
+\frac{1}{2} \dual{\operator(\m(t))}{\m(t)}_{\Omega_1}
- \dual{\hext}{\m(t)}_{\Omega_1}
\end{align}
satisfies
\begin{align}\label{eq:energy_est}
\EE\left(\m(t)\right) + \alpha\norm{\partial_t\m}{\L^2(\Omega_t)}^2 \leq \EE\left(\m_0\right)
\end{align}
for almost every $t \in (0,T)$.
\end{propappendix}
\begin{proof}
Given an arbitrary $t \in (0,T)$, let $j=0,\dots,N-1$ such that $t \in [t_j,t_{j+1})$.
Let $i=0,\dots,j$.
From the stability estimate~\eqref{eq:nabla_m_bounded}, we get
\begin{align*}
\EE(\m_h^{i+1}) - \EE(\m_h^i)
& \leq - \alpha k \norm{\v_h^i}{\L^2(\Omega_1)}^2
-\Cexchange(\theta - 1/2) k^2 \norm{\nabla \v_h^i}{\L^2(\Omega_1)}^2 \\
& \quad \underbrace{+\frac{1}{2} \dual{\operator(\m_h^{i+1})}{\m_h^{i+1}}_{\Omega_1}
-\frac{1}{2} \dual{\operator(\m_h^i)}{\m_h^i}_{\Omega_1}
- k \dual{\operator_h(\m_h^i)}{\v_h^i}_{\Omega_1}}_{=:T_1} \\
& \quad \underbrace{- \dual{\hext}{\m_h^{i+1}}_{\Omega_1}
+ \dual{\hext}{\m_h^i}_{\Omega_1}
+ k \dual{\hext_h}{\v_h^i}_{\Omega_1}}_{=:T_2}.
\end{align*}
Since $\operator$ is linear and self-adjoint, straightforward calculations show
\begin{equation*}
\begin{split}
T_1
& = k \dual{\operator(\m_h^i)-\operator_h(\m_h^i)}{\v_h^i}_{\Omega_1}
+ \frac{1}{2} k \dual{\operator(\m_h^{i+1}-\m_h^i)}{\v_h^i}_{\Omega_1} \\
& \quad +\frac{1}{2} \dual{\operator(\m_h^{i+1}+\m_h^i)}{\m_h^{i+1}-\m_h^i-k \v_h^i}_{\Omega_1},
\end{split}
\end{equation*}
and
\begin{equation*}
T_2 = - k \dual{\hext-\hext_h}{\v_h^i}_{\Omega_1}
- \dual{\hext}{\m_h^{i+1}-\m_h^i-k \v_h^i}_{\Omega_1}.
\end{equation*}
Combining the Cauchy-Schwarz inequality with Lemma~\ref{lemma1:aux}, Lemma~\ref{lemma3:aux}, and the $\L^2$-stability of $\operator$, we get 
\begin{equation*}
k \left\vert \dual{\operator(\m_h^{i+1}-\m_h^i)}{\v_h^i} \right\vert
\lesssim k \norm{\m_h^{i+1}-\m_h^i}{\L^2(\Omega_1)} \norm{\v_h^i}{\L^2(\Omega_1)}
\lesssim k^2 \norm{\v_h^i}{\L^2(\Omega_1)}^2.
\end{equation*}
The H\"older inequality, together with assumption~\eqref{eq:boundedness_l4}, Lemma~\ref{lemma2:aux}, and Lemma~\ref{lemma3:aux} yields
\begin{equation*}
\begin{split}
& \left\vert \dual{\operator(\m_h^{i+1}+\m_h^i)}{\m_h^{i+1}-\m_h^i-k \v_h^i} \right\vert \\
& \quad \leq \c{l4_boundedness} \norm{\m_h^{i+1}+\m_h^i}{\L^4(\Omega_1)} \norm{\m_h^{i+1}-\m_h^i-k \v_h^i}{\L^{4/3}(\Omega_1)} \\
& \quad \lesssim k^2 \norm{\m_h^{i+1}+\m_h^i}{\L^4(\Omega_1)} \norm{\v_h^i}{\L^{8/3}(\Omega_1)}^2 \\
& \quad \lesssim k^2 \norm{\v_h^i}{\L^{8/3}(\Omega_1)}^2.
\end{split}
\end{equation*}
The same argument also shows
\begin{equation*}
\left\vert\dual{\hext}{\m_h^{i+1}-\m_h^i-k \v_h^i}\right\vert
\leq \norm{\hext}{\L^4(\Omega_1)} \norm{\m_h^{i+1}-\m_h^i-k \v_h^i}{\L^{4/3}(\Omega_1)}
\lesssim k^2 \norm{\v_h^i}{\L^{8/3}(\Omega_1)}^2.
\end{equation*}
The log-convexity of Lebesgue norms and the Sobolev embedding $\H^1(\Omega_1) \subset \L^4(\Omega_1)$ yield
\begin{equation*}
\norm{\v_h^i}{\L^{8/3}(\Omega_1)}^2
\lesssim \norm{\v_h^i}{\L^2(\Omega_1)} \norm{\v_h^i}{\L^4(\Omega_1)}
\lesssim \norm{\v_h^i}{\L^2(\Omega_1)} \norm{\v_h^i}{\H^1(\Omega_1)}.
\end{equation*}
Altogether, we thus obtain
\begin{equation*}
\begin{split}
& \EE(\m_h^{i+1}) - \EE(\m_h^i)
+ \alpha k \norm{\v_h^i}{\L^2(\Omega_1)}^2
- k \dual{\operator(\m_h^i)-\operator_h(\m_h^i)}{\v_h^i}_{\Omega_1}
+ k \dual{\hext-\hext_h}{\v_h^i}_{\Omega_1}\\
& \quad \lesssim k^2 \left(\norm{\v_h^i}{\L^2(\Omega_1)}^2 + \norm{\v_h^i}{\L^2(\Omega_1)} \norm{\v_h^i}{\H^1(\Omega_1)}\right).
\end{split}
\end{equation*}
Analogously to~\eqref{eq:mhk-}, we define $\m_{hk}^+\in\PP^0(\II_k;\VV_h)$ by $\m_{hk}^+(t):=\m_h^{i+1}$ for $t_i \le t < t_{i+1}$.
Arguing as in Lemma~\ref{lem:subsequences}, one proves that $\m_{hk}^+\to\m$ strongly in $\L^2(\Omega_T)$ for a subsequence.
Summing the last estimate over $i=0,\dots,j$, we obtain
\begin{equation*}
\begin{split}
& \EE(\m_{hk}^+(t)) - \EE(\m_h^0)
+ \alpha \norm{\v_{hk}^-}{\L^2(\Omega_t)}^2 \\
& \quad - \dual{\operator(\m_{hk}^-)-\operator_h(\m_{hk}^-)}{\v_{hk}^-}_{\Omega_{t_{j+1}}}
+ \dual{\hext-\hext_h}{\v_{hk}^-}_{\Omega_{t_{j+1}}} \\
& \qquad \lesssim k \left(\norm{\v_{hk}^-}{\L^2(\Omega_{t_{j+1}})}^2
+ \norm{\v_{hk}^-}{\L^2(\Omega_{t_{j+1}})} \norm{\nabla\v_{hk}^-}{\L^2(\Omega_{t_{j+1}})}\right).
\end{split}
\end{equation*}
Exploiting the available convergence results on $\m_{hk}^{\pm}$ and $\v_{hk}^{-}$, the boundedness of $\sqrt{k}\norm{\nabla \v_{hk}^-}{\L^2(\Omega_T)}$ and $\norm{\v_{hk}^-}{\L^2(\Omega_T)}$ from Lemma~\ref{lemma:dpr}, and assumptions~\eqref{assumption:chi:convergence2}--\eqref{assumption:f2}, we can use standard arguments with lower semicontinuity for the limit $(h,k) \to (0,0)$ and derive the desired result~\eqref{eq:energy_est}.
\end{proof}
\begin{remarkappendix}
The operator $\operator$ is linear, $\L^2$-bounded and self-adjoint in many concrete situations, e.g., when it comprises the uniaxial anisotropy contribution from Section~\ref{sec:anisotropy} and the stray field contribution.
In this case $\operator$ is also well-defined and bounded as operator $\operator: \L^p(\Omega_1) \to \L^p(\Omega_1)$ for all $1 < p < \infty$, see Ref.~\refcite{praetorius2004}, and Assumption~\eqref{eq:boundedness_l4} is therefore satisfied.
Assumptions~\eqref{assumption:chi:convergence2} and~\eqref{assumption:f2} are slightly stronger than~\eqref{assumption:chi:convergence} and~\eqref{assumption:f}, respectively. However, they are fulfilled in many actual realizations $\operator_h$ and $\hext_h$, see Section~\ref{section:fredkinkoehler} and Remark~\ref{rem:f}.
\end{remarkappendix}

\section*{Acknowledgements}
The authors acknowledge financial support through the WWTF project MA09-029, the FWF project P21732, the FWF project SFB-ViCoM F4112-N13, the FWF graduate school W1245, and the innovative projects initiative of Vienna University of Technology.

\end{document}